\documentclass[reqno]{amsart}

\usepackage[all]{xy}
\usepackage{graphicx}

\usepackage{amssymb}
\usepackage{amsmath}
\usepackage{mathrsfs}
\usepackage{epsfig}
\usepackage{amscd}
\usepackage{graphicx, color}

\def\E{\ifmmode{\mathbb E}\else{$\mathbb E$}\fi} %natural numbers
\def\N{\ifmmode{\mathbb N}\else{$\mathbb N$}\fi} %natural numbers%
\def\R{\ifmmode{\mathbb R}\else{$\mathbb R$}\fi} %real numbers
\def\Q{\ifmmode{\mathbb Q}\else{$\mathbb Q$}\fi} %rational numbers
\def\C{\ifmmode{\mathbb C}\else{$\mathbb C$}\fi} %complex numbers
\def\H{\ifmmode{\mathbb H}\else{$\mathbb H$}\fi} %complex numbers
\def\Z{\ifmmode{\mathbb Z}\else{$\mathbb Z$}\fi} %integers
\def\P{\ifmmode{\mathbb P}\else{$\mathbb P$}\fi} %real numbers
\def\T{\ifmmode{\mathbb T}\else{$\mathbb T$}\fi} %real numbers
\def\SS{\ifmmode{\mathbb S}\else{$\mathbb S$}\fi} %real numbers
\def\DD{\ifmmode{\mathbb D}\else{$\mathbb D$}\fi} %real numbers

\newcommand{\e}{\varepsilon}

\newcommand{\del}{\partial}
\newcommand{\Cont}{{\operatorname{Cont}}}

\newcommand{\ben}{\begin{enumerate}}
\newcommand{\een}{\end{enumerate}}
\newcommand{\be}{\begin{equation}}
\newcommand{\ee}{\end{equation}}
\newcommand{\bea}{\begin{eqnarray}}
\newcommand{\eea}{\end{eqnarray}}
\newcommand{\beastar}{\begin{eqnarray*}}
\newcommand{\eeastar}{\end{eqnarray*}}
\newcommand{\bc}{\begin{center}}
\newcommand{\ec}{\end{center}}

\theoremstyle{theorem}
\newtheorem{thm}{Theorem}[section]
\newtheorem{cor}[thm]{Corollary}
\newtheorem{lem}[thm]{Lemma}
\newtheorem{prop}[thm]{Proposition}

\theoremstyle{definition}
\newtheorem{defn}[thm]{Definition}
\newtheorem{rem}[thm]{Remark}

\newtheorem{notation}[thm]{\rm\bfseries{Notation}}

\newtheorem*{thm*}{Theorem}

\numberwithin{equation}{section}

\def\R{{\mathbb R}}

\def\Crit{{\hbox{Crit}}}

\def\E{{\mathbb E}}
\def\Z{{\mathbb Z}}
\def\C{{\mathbb C}}
\def\R{{\mathbb R}}
\def\P{{\mathbb P}}

\def\N{{\mathbb N}}

\def\11{{\mathbb I}}

\def\delbar{{\overline \partial}}

\def\Fix{{\text{\rm Fix}}}

\def\id{{\text{\rm id}}}

\def\C{\mathbb{C}}
\def\Z{\mathbb{Z}}

\def\T{\mathbb{T}}

\def\Q{\mathbb{Q}}

\def\E{\ifmmode{\mathbb E}\else{$\mathbb E$}\fi} %natural numbers
\def\N{\ifmmode{\mathbb N}\else{$\mathbb N$}\fi} %natural numbers
\def\R{\ifmmode{\mathbb R}\else{$\mathbb R$}\fi} %real numbers
\def\Q{\ifmmode{\mathbb Q}\else{$\mathbb Q$}\fi} %rational numbers
\def\C{\ifmmode{\mathbb C}\else{$\mathbb C$}\fi} %complex numbers
\def\H{\ifmmode{\mathbb H}\else{$\mathbb H$}\fi} %complex numbers
\def\Z{\ifmmode{\mathbb Z}\else{$\mathbb Z$}\fi} %integers
\def\P{\ifmmode{\mathbb P}\else{$\mathbb P$}\fi} %real numbers
\def\SS{\ifmmode{\mathbb S}\else{$\mathbb S$}\fi} %real numbers
\def\DD{\ifmmode{\mathbb D}\else{$\mathbb D$}\fi} %real numbers

\def\R{{\mathbb R}}

\def\Crit{{\hbox{Crit}}}
\def\E{{\mathbb E}}
\def\Z{{\mathbb Z}}
\def\C{{\mathbb C}}
\def\R{{\mathbb R}}

\def\N{{\mathbb N}}
\def\LL{{\mathcal L}}
\def\CL{{\mathcal L}}
\def\MM{{\mathcal M}}

\def\JJ{{\mathcal J}}

\def\FF{{\mathcal F}}

\def\ev{{\text{\rm ev}}}

\def\delbar{{\overline \partial}}
\def\Int{{\text{\rm Int}}}

\def\e{\varepsilon}

\def\CA{{\mathcal A}}
\def\CB{{\mathcal B}}
\def\CC{{\mathcal C}}
\def\CD{{\mathcal D}}

\def\CF{{\mathcal F}}

\def\CH{{\mathcal H}}

\def\CJ{{\mathcal J}}

\def\CL{{\mathcal L}}
\def\CM{{\mathcal M}}
\def\CN{{\mathcal N}}

\def\CP{{\mathcal P}}

\def\CP{{\mathcal P}}

\def\CV{{\mathcal V}}
\def\CW{{\mathcal W}}

\def\supp{\operatorname{supp}}

\def\End{\operatorname{End}}

\def\coker{\operatorname{Coker}}
\def\span{\operatorname{span}}

\def\Image{\operatorname{Image}}

\def\Ev{\text{\rm Ev}}

\def\mq{\mathfrak{q}}
\def\mp{\mathfrak{p}}
\def\mH{\mathfrak{H}}
\def\mh{\mathfrak{h}}
\def\ma{\mathfrak{a}}
\def\ms{\mathfrak{s}}
\def\mm{\mathfrak{m}}
\def\mn{\mathfrak{n}}
\def\mz{\mathfrak{z}}
\def\mw{\mathfrak{w}}
\def\Hoch{{\tt Hoch}}
\def\mt{\mathfrak{t}}
\def\ml{\mathfrak{l}}
\def\mT{\mathfrak{T}}
\def\mL{\mathfrak{L}}
\def\mg{\mathfrak{g}}
\def\md{\mathfrak{d}}
\def\mr{\mathfrak{r}}

\def\End{\operatorname{End}}
\def\dim{\operatorname{dim}}

\def\supp{\operatorname{supp}}

\def\End{\operatorname{End}}

\def\coker{\operatorname{Coker}}
\def\span{\operatorname{span}}

\def\Cont{\operatorname{Cont}}
\def\Crit{\operatorname{Crit}}
\def\Spec{\operatorname{Spec}}
\def\Sing{\operatorname{Sing}}
\def\GFQI{\frak{G}}
\def\Index{\operatorname{Index}}

\def\Image{\operatorname{Image}}
\def\ev{\operatorname{ev}}
\def\Int{\operatorname{Int}}

\def\ben{\begin{enumerate}}
\def\een{\end{enumerate}}
\def\be{\begin{equation}}
\def\ee{\end{equation}}
\def\bea{\begin{eqnarray}}
\def\eea{\end{eqnarray}}
\def\beastar{\begin{eqnarray*}}
\def\eeastar{\end{eqnarray*}}
\def\bc{\begin{center}}
\def\ec{\end{center}}

\begin{document}

\quad \vskip1.375truein

\def\mq{\mathfrak{q}}
\def\mp{\mathfrak{p}}
\def\mH{\mathfrak{H}}
\def\mh{\mathfrak{h}}
\def\ma{\mathfrak{a}}
\def\ms{\mathfrak{s}}
\def\mm{\mathfrak{m}}
\def\mn{\mathfrak{n}}
\def\mz{\mathfrak{z}}
\def\mw{\mathfrak{w}}
\def\Hoch{{\tt Hoch}}
\def\mt{\mathfrak{t}}
\def\ml{\mathfrak{l}}
\def\mT{\mathfrak{T}}
\def\mL{\mathfrak{L}}
\def\mg{\mathfrak{g}}
\def\md{\mathfrak{d}}
\def\mr{\mathfrak{r}}
\def\Cont{\operatorname{Cont}}
\def\Crit{\operatorname{Crit}}
\def\Spec{\operatorname{Spec}}
\def\Sing{\operatorname{Sing}}
\def\GFQI{\text{\rm g.f.q.i.}}
\def\Index{\operatorname{Index}}
\def\Cross{\operatorname{Cross}}
\def\Ham{\operatorname{Ham}}
\def\Fix{\operatorname{Fix}}
\def\Graph{\operatorname{Graph}}
\def\id{\text\rm{id}}

\title[Fredholm theory of contact instantons]
{Bordered contact instantons and their Fredholm theory and generic transversalities}

\author{Yong-Geun Oh}
\address{Center for Geometry and Physics, Institute for Basic Science (IBS),
77 Cheongam-ro, Nam-gu, Pohang-si, Gyeongsangbuk-do, Korea 37673
\& POSTECH, Gyeongsangbuk-do, Korea}
\email{yongoh1@postech.ac.kr}
\thanks{This work is supported by the IBS project \# IBS-R003-D1}

\begin{abstract}  In this article, we first establish the Fredholm theory
 for the bordered contact instantons  defined on the punctured Riemann surfaces
 with prescribed asymptotic  condition near the boundary punctures.
 We then prove the generic mapping transversality under the perturbation of
 Legendrian boundary condition. We also establish their generic (0-jet) evaluation
 transversality results of their moduli space under the perturbations of
 CR almost complex structures and of Legendrian boundary conditions. These
 are fundamental ingredients of the construction of the moduli space of
 bordered contact instantons and their applications.
\end{abstract}


\maketitle

\begin{center}
In memory of Bumsig Kim
\end{center}

\tableofcontents

\section{Introduction}

A contact manifold $(M,\xi)$ is a $2n+1$ dimensional manifold
equipped with a completely non-integrable distribution of rank $2n$,
called a contact structure. Complete non-integrability of $\xi$
can be (locally) expressed by the non-vanishing property
$$
\lambda \wedge (d\lambda)^n \neq 0
$$
for a one-form $\lambda$. When $\xi$ is coorientable, we can choose such a one-form globally
so that $\ker \lambda = \xi$ is called a contact form associated to $\xi$.
In the present article, we will always assume that $(M,\xi)$ is coorientable equipped with
a coorientation without mentioning further.

Each contact form $\lambda$ of $\xi$ canonically induces a splitting
$$
TM = \R\langle R_\lambda\rangle \oplus \xi.
$$
Here $R_\lambda$ is the Reeb vector field of $\lambda$,
which is uniquely determined by the equations
$$
R_\lambda \rfloor \lambda \equiv 1, \quad R_\lambda \rfloor d\lambda \equiv 0.
$$
We denote by $\Pi=\Pi_\lambda: TM \to TM$ the idempotent, i.e., an endomorphism satisfying
$\Pi^2 = \Pi$ such that $\ker \Pi = \R \langle R_\lambda\rangle$ and $\operatorname{Im} \Pi = \xi$.
Denote by $\pi=\pi_\lambda: TM \to \xi$ the associated projection.

\begin{defn} Let $(M,\xi)$ be a contact manifold and let $\lambda$ be a contact form of $\xi$.
Let $J \in \End(TM)$ be an endomorphism satisfying $J^2 = - \Pi$ such that
$d\lambda(\cdot,J\cdot)$ is positive definite on $\xi$. We say that such $J$ is \emph{adapted to} $\lambda$.
 We define the set
\be\label{eq:JJQlambda}
\CJ_\lambda(M,\xi) = \{J :\xi \to \xi \mid J^2 = - \Pi, \, J \, \text{adapted to } \, \lambda\}
\ee
\end{defn}
Following \cite{oh-wang:CR-map1}, we call any such triple $(M,\lambda,J)$ a contact triad
 of $(M,\xi)$.
For each given contact triad, we equip $M$ with the triad metric
$$
g = d\lambda(\cdot, J \cdot) + \lambda \otimes \lambda.
$$

We denote by $\Sigma$ be a compact Riemann surface $(\Sigma, j)$ with or without
boundary, and denote by $\dot \Sigma$ a punctured Riemann surface with
a finite number of punctures which may be either from the interior or from the boundary of $\Sigma$.

For a given map $w: \dot \Sigma \to M$,
we can decompose its derivative
$du$, regarded as a $w^*TM$-valued one-form on $\dot \Sigma$, into
\be\label{eq:du}
dw = d^\pi w + w^*\lambda \otimes R_\lambda
\ee
 where $d^\pi w := \pi dw$. Furthermore $d^\pi w$ is decomposed into
\be\label{eq:dpiu}
d^\pi w = \delbar^\pi w + \del^\pi w
\ee
where $\delbar^\pi w: = (dw^\pi)_J^{(0,1)}$ (resp. $\del^\pi w: = (dw^\pi)_J^{(1,0)}$) is
the anti-complex linear part (resp. the complex linear part) of $d^\pi w: (T\dot \Sigma, j) \to (\xi,J|_\xi)$.
(For the simplicity of notation, we will abuse our notation by often denoting $J|_\xi$ by $J$.
We also simply write $((\cdot)^\pi)_J^{(0,1)} = (\cdot)^{\pi(0,1)}$ and
 $((\cdot)^\pi)_J^{(1,0)} = (\cdot)^{\pi(1,0)}$ in general, unless there is a reason to
 specify $J$ in notation.)

\begin{defn}[Contact instanton] Let $\Sigma$ be as above.
We call a pair of $(j,w)$ of a complex structure on $\Sigma$ and a map $w:\dot \Sigma \to M$ a
a \emph{contact Cauchy-Riemann map} if $\delbar^\pi w = 0$, and a
\emph{contact instanton} if it satisfies
\be\label{eq:closed}
\delbar^\pi w = 0, \, \quad d(w^*\lambda \circ j) = 0.
\ee
\end{defn}

To avoid notional complexity and for the simplicity of exposition, we will assume that $\Sigma$ is
a compact surface of genus zero with or without boundary. We denote a marked Riemann surface by
$$
(\Sigma, (z_1, \ldots, z_k))
$$
where $(z_1, \ldots, z_k)$ are boundary marked points, unless said otherwise. They are ordered counterclockwise.
We denote by $\overline{z_iz_{i+1}}$ the arc segment between $z_i$ and $z_{i+1}$

In \cite{oh:contacton-Legendrian-bdy}, the present author introduced the open-string version, the
boundary value problem
\be\label{eq:contacton-Legendrian-bdy-intro}
\begin{cases}
\delbar^\pi w = 0, \, \quad d(w^*\lambda \circ j) = 0\\
w(\overline{z_iz_{i+1}}) \subset R_i, \quad i = 1, \ldots, k
\end{cases}
\ee
for a map $w:(\dot \Sigma, \del \dot \Sigma) \to (M,\vec R)$ for the Legendrian boundary
condition
$$
\vec R = \{R_1, \cdots, R_k\},
$$
with a suitable asymptotic boundary condition at the punctures,
and established its ellipticity by deriving relevant a priori
estimates. (See Subsection \ref{subsec:off-shell} for the precise description of the asymptotic
boundary condition.)

The Fredholm theory of contact instantons of closed-string version has been
established by the present author in \cite{oh:contacton}.
One of the main purposes of the present article is to extend the story to the open-string case and
establish the Fredholm theory for the bordered contact instantons. We also
establish the necessary generic transversality results of contact instantons under the perturbations of
\begin{enumerate}
\item contact forms $\lambda$,
\item the adapted CR almost
complex structure $J$ and
\item  the Legendrian boundary condition $R_i$'s.
\end{enumerate}
These are fundamental analytical ingredients needed for the applications to construct the moduli space of contact instantons with
prescribed asymptotic limits and to establish  the gluing theorem for the contact instanton
Floer trajectories similarly as in the Floer theory of Lagrangian intersections
\cite{floer-unregularized} under the change of $J$'s and \cite{oh:fredholm} under
the change of boundary condition.

\begin{rem} We also need to make these studies for the
Hamiltonian-perturbed contact instantons
\be\label{eq:perturbed-contacton-bdy-intro}
\begin{cases}
(du - X_H \otimes dt)^{\pi(0,1)} = 0, \quad d(e^{g_H(u)}(u^*\lambda + H\, dt)\circ j) = 0\\
u(\tau,0) \in R_0, \quad u(\tau,1) \in R_1
\end{cases}
\ee
under the change of $H$ too: Here
the function $g_H: \R\times [0,1] \to \R$ is some canonically defined function
associated to $H$. (See \cite{oh:contacton-Legendrian-bdy}.)
This equation is the contact counterpart of the celebrated Floer's Hamiltonian-perturbed
Cauchy-Riemann equation in symplectic geometry \cite{floer:fixedpoints}. Similarly as in symplectic geometry,
such an extension is an easy generalization of the arguments employed in the
present article in that the presence of $H$ does not play
much role and so omitted for the clarity and simplicity of the exposition.
\end{rem}

\subsection{Fredholm theory of moduli spaces
of bordered contact instantons}

To develop a relevant Fredholm theory of the moduli spaces of
 bordered contact instantons with Legendrian boundary,
we closely follow that of \cite{oh:contacton} by incorporating the
boundary condition in the off-shell function spaces. For this purpose, we
also need to establish  all generic transversality results
of the Reeb chords $\gamma^\pm$ and of the moduli space
$\CM(M,\lambda, J;R;\gamma^-,\gamma^+)$.
Since such a transversality result under the perturbation of boundary Legendrian submanifolds
are not considered in the general situation before, especially in relation to the
present context of bordered contact instantons, we give their full details
in Appendix \ref{sec:nondegeneracy-chords} for the Reeb chords
under the perturbation of contact forms (see \cite[Appendix]{ABW} for the proof of this generic nondegeneracy
for the closed Reeb orbits), and in Part I  for the moduli space of contact instantons
under the perturbation of CR almost complex structures and under that of  Legendrian boundary conditions,
respectively.

One point we would like to highlight in the study of generic nondegeneracy
of Reeb chords is that we  consider the chords \emph{in the sense of Moore paths}
whose  domains vary and so whose elements are represented by the pairs $(\gamma,T)$ 
such that
\be\label{eq:Moore-path}
T \in \R, \quad \gamma:[0,T] \to M, \quad T = \int \gamma^*\lambda
\ee
including $T = 0$. We emphasize here that we include the zero period
$T= 0$ and consider the constant paths $(0,\gamma)$.
Such a path exists only in the non-generic situation when $\psi(R)$ intersects $R$, e.g.,
when $\psi = id$. This transversality is important for the calculation of
contact instanton homology constructed in \cite{oh:entanglement1}, \cite{oh:contacton-gluing}
and its applications  \cite{oh:entanglement1}, \cite{oh:shelukhin-conjecture}
to Sandon-Shelukhin type quantitative contact topology
\cite{sandon:translated}, \cite{shelukhin:contactomorphism}.

\begin{rem} It is an interesting open problem to equip
Kuranishi structures on the compactified moduli spaces of contact instantons
which are suitably
compatible so that they give rise to the (Legendrian) contact DGA that appears in the
case of trivial symplectic cobordism, i.e., the case of symplectization of contact manifolds.
(See \cite{bao-honda-kuranishi}, \cite{pardon-contact}, \cite{ishikawa}.) We believe that
the general abstract framework of the Kuranishi structure from \cite{fooo:book-kuranishi}
or some variation of that of \cite{pardon-contact}, \cite{bao-honda-kuranishi} applies to
the current case of contact instantons too. We hope to come back to this elsewhere.
\end{rem}
We will closely follow the off-shell analytical framework
from  \cite{oh:contacton}, \cite{oh-savelyev} which handle the context
of closed strings. In particular in terms of the decomposition $d\pi = d^\pi w + w^*\lambda\, R_\lambda$
and $Y = Y^\pi + \lambda(Y) R_\lambda$, an explicit tensorial formula of the linearized operator
$D\Upsilon(w)$ is derived in \cite[Theorem 10.1]{oh:contacton} which is the starting point of the
Fredholm theory.
(See also \cite[Theorem 1.15]{oh-savelyev} and Theorem \ref{thm:linearization} in the present paper.)
We note that there are three kinds of perturbations we can think of as mentioned above, i.e., $J$, $\lambda$ and
the boundary Legendrian submanifolds $R_i$'s. The study of perturbations of $J$ is given in
\cite{oh:contacton} which is of no change in the present open string case, and perturbation of
contact forms is largely subsumed into that of $J$. \emph{Therefore we will
focus on the perturbation of the boundary in the present paper after establishment of the Fredholm
theory for the bordered contact instantons.}

For this purpose, similarly as in symplectic geometry
\cite{oh:fredholm}, we consider the universal moduli space
$$
\MM^{\text{\rm univ}}(M,\lambda;\overline \gamma, \underline \gamma)
$$
consisting of the triples $(w,(J,\vec R))$ satisfying \eqref{eq:contacton-Legendrian-bdy-intro}
with $\vec R = \{R_i\}_{i=1}^k$ \emph{with fixed asymptotics
$(\overline \gamma, \underline \gamma)$ at the punctures of $\dot \Sigma$}.
Denote by $\mathfrak{Led}(M,\xi)$ the set of smooth Legendrian
submanifolds the set of whose Reeb chords for this given fixed asymptotics.

We regard the assignment
$$
\Upsilon^{\text{\rm univ}}: w \mapsto \left(\delbar^\pi_J w, d(w^*\lambda \circ j)\right), \quad
\Upsilon: = (\Upsilon_1,\Upsilon_2)
$$
as a section of the (infinite dimensional) vector bundle
\be\label{eq:universal-CD}
\CC\CD^{\text{\rm univ}} \to \CF^{\text{\rm univ}}(M,\lambda; \overline \gamma, \underline \gamma)
\ee
where we put
$$
\CF^{\text{\rm univ}}(M,\lambda; \overline \gamma, \underline \gamma): =
\bigcup_{\vec R \in \mathfrak{Led}(M,\xi)} \CF\left(M,\lambda;\vec R; \overline \gamma, \underline \gamma\right)
\times  \CJ_\lambda(M,\xi).
$$
We denote by $\Upsilon^{\text{\rm univ}}$ the parameterized section of \eqref{eq:universal-CD}
defined by
$$
\Upsilon^{\text{\rm univ}}((w,J),\vec R) = \left(\delbar^\pi_J w, d(w^*\lambda \circ j)\right).
$$
Here $\CF(M,\lambda; \vec R; \overline \gamma, \underline \gamma)$ is the off-shell function space
associated to the moduli space $\MM(M,\lambda;J;\vec R; \overline \gamma, \underline \gamma)$.
(See Definition \ref{defn:Banach-manifold}.)
We refer readers to Notation \ref{nota:CD} for the definition of $\CC\CD$.
Then we summarize the main Fredholm results established in the present article into the following.
This is the open-string counterpart of the Fredholm result
established in \cite{oh:contacton} for the closed string case.

\begin{thm}\label{thm:trans-intro} Let $\ell > 0$ be a given sufficiently large integer.
Then
\begin{enumerate}
\item $\Upsilon^{\text{\rm univ}}$ is a smooth submersion on the open subset of $\CF^{\text{\rm univ}}$
consisting of somewhere injective map $w$.
\item The parameterized moduli space
$\MM^{\CJ_\lambda}(M,\lambda, \vec R;\overline \gamma, \underline \gamma)$ over $\CJ_\lambda(M,\xi)$
with $\vec R$ fixed is an infinite dimensional $C^\ell$ Banach manifold.
\item The projection
$$
 (\Upsilon^{\CJ_\lambda})^{-1}(0) \to \CJ_\lambda(M,\xi)
$$
(with $\vec R$ fixed) is a Fredholm map and its index is the same as that of $D\Upsilon(w)$
for a (and so any) $w \in \MM(M,\lambda, \vec R; J;\overline \gamma, \underline \gamma)$.
\end{enumerate}
\end{thm}

\begin{notation} We will denote by $\CM^{(\cdot)}$ or $\CM(\cdots; (\cdot))$
for various \emph{parameterized} moduli spaces
over the parameter space $(\cdot)$ such as $\CJ_\lambda = \CJ_\lambda(M,\xi)$, $\mathfrak{Leg}(M,\xi)$
and others. We apply similar notations for the associated off-shell function spaces. We also denote by
 $\Upsilon^{(\cdot)}$ the associated parameterized section of $\Upsilon$.
\end{notation}

\subsection{Transversality under the perturbation of boundary condition}

In \cite{oh:fredholm}, the present author established a transversality result
of the open string version of the Gromov-Witten-Floer theory under the
perturbation of Lagrangian boundary conditions.
We also need to study the transversality result under the perturbation of
Legendrian boundaries for the construction of Fukaya-type category on contact manifolds
\cite{oh:entanglement2}. Such a study is carried out by Ekholm-Etnyre-Sullivan \cite{EES-jdg}
in the study of Chekanov-Eliashberg DGA of Legendrian submanifolds through
symplectization adapting that of \cite{oh:fredholm}.

In the present section, we develop
the Legendrian counterpart for the contact instantons imitating the arguments
used in \cite{oh:fredholm}. Although the overall scheme of the proof largely follows
that of \cite{oh:fredholm} its details are much more subtle and nontrivial.
This is largely because the nature of contact instanton equation is more
complicated than the pseudoholomorphic curves, especially because
the study of adjoint problem of the linearized equation involves much more
nontrivial systematic tensorial calculations than that of \cite{oh:fredholm}.
We attract readers' attention that this is another place where
\emph{the framework of contact instantons
exhibits its naturality and compatibility with the existing contact geometry},
in that the contact distribution component and the Reeb component of the system
\eqref{eq:contacton-Legendrian-bdy-intro} well interact with each other
through the optimal covariant tensor calculus leading to the proof of
the parametric transversality result under the perturbation of boundaries.
(See Section \ref{sec:adjoint-eq} for such tensorial calculations.)

As in the proof of \cite[Theorem I]{oh:fredholm}, we transforming the problem
of perturbing boundaries by ambient contact isotopies, and study the following
fibration
$$
\Pi^{\CL eg}: \CM^{\CL eg}(M,\lambda;\overline \gamma, \underline \gamma) \to
\mathfrak{Led}(M,\xi)
$$
with $J$ fixed. We then prove the following generic transversality result.

\begin{thm}[Theorem \ref{thm:trans-under-bdy}]\label{thm:trans-under-bdy-intro}
 Let $(M,\xi)$ be a contact manifold, and let
$\lambda$ a contact form  be given. We consider \eqref{eq:contacton-Legendrian-bdy-intro}.
Fix $J$ and $k$ and consider transversal Legendrian link $\vec R$.
Then there exists a residual subset of $J$'s with $\vec R$ fixed
(resp. of $\vec R = (R_1, \ldots, R_k)$
of Legendrian submanifolds with $J$ fixed) such that  the moduli space
$\MM(M,\lambda, \vec R;\overline \gamma, \underline \gamma) $ is transversal.
\end{thm}
We refer to Appendix \ref{sec:nondegeneracy-chords} for the definition of \emph{transversal link}.

\subsection{Generic evaluation transversality}
\label{subsec:0jet-eval-transversality-intro}

Another crucial general analytical ingredient is
the \emph{evaluation map transversality}.
Such an evaluation transversality will be important for
the application to contact topology, for example,
in the proof of Shelukhin's conjecture \cite{oh:shelukhin-conjecture} and
in the construction of Fukaya-type category of
contact manifolds generated by Legendrian submanifolds in \cite{oh:entanglement2},
similarly as in the study of pseudoholomorphic curves in symplectic geometry.
A rigorous proof of the evaluation transversality is rather subtle even in the pseudoholomorphic
curve theory as  already mentioned in \cite{oh:highjet}.
A conceptually canonical proof of the evaluation map transversality is
given in \cite[Section 10.5]{oh:book1}, which in turn followed the scheme
provided by Le and Ono \cite{le-ono:perturbation} and Zhu and the author \cite{oh-zhu:ajm}
in their studies of one-jet evaluation transversality which
is based on a standard structure theorem of the distributions with point support.
(See Theorem \ref{thm:gelfand} below.)
Naturality of the proof in \cite{oh-zhu:ajm}, \cite{oh:highjet} enables us to adapt it
to the current context of contact instantons.  However the proof of the evaluation transversality
study for contact instantons is significantly more nontrivial in its details
than the case of pseudoholomorphic curves thanks to the different nature
of the equation which involves a system of partial differential equations of mixed degree.

We first recall the off-shell setting of the study of evaluation transversality.
For given Legendrian link $\vec R = (R_1,\cdots, R_k)$, we consider the moduli space
$$
\CM((\dot \Sigma,\del \dot \Sigma),(M,\vec R);J)
$$
 of finite energy maps $w: \dot \Sigma \to M$ satisfying the equation
\eqref{eq:contacton-Legendrian-bdy-intro} as before. (We refer readers \cite{oh:entanglement1,oh:shelukhin-conjecture}
for the definition of relevant energies.)

We will treat the two cases,  evaluation at an interior marked point and one at a boundary marked point,
separately. We denote by the subindex $(\ell,k)$ the number of interior and boundary marked points respectively.
Consider the parameterized moduli space
\beastar
&{}& \CM_{(1,0)}((\dot \Sigma, \del \dot \Sigma), (M,\vec R);\CJ_\lambda) \\
& = & \{((j,w),J, z) \mid w: \Sigma \to
M, \, \Upsilon(J,(j,w))  = 0,\, \, w(\del \dot \Sigma) \subset \vec R,\,  z \in \Int \dot \Sigma \}.
\eeastar
The evaluation map $\ev^+: \CM_{(1,0)}((\dot \Sigma,\del \dot \Sigma,(M, \vec R)) \to M$ is defined by
$$
\ev^+((j,w),z) = w(z).
$$
We then have the fibration
$$
\widetilde \CM_{(1,0)} ((\dot \Sigma,\del \dot \Sigma),(M, \vec R);\CJ_\lambda)  = \bigcup_{J \in \CJ_\lambda}
\widetilde \CM_{(1,0)} ((\dot \Sigma,\del \dot \Sigma),(M, \vec R);J)
\to \CJ_\lambda
$$
and
$$
\widetilde \CM_{(1,0)}^{\text{\rm inj}}((\dot \Sigma,\del \dot \Sigma),(M, \vec R))
$$
to be the open subset of
$\widetilde \CM_{(1,0)}(M,\lambda, \vec R)$ consisting of somewhere injective
contact instantons.
We have the universal ($0$-jet) evaluation map
$$
\Ev^+: \widetilde \CM_{(1,0)} ((\dot \Sigma,\del \dot \Sigma),(M, \vec R);\CJ_\lambda) \to M.
$$
We also consider the boundary evaluation map
$$
\Ev_\del: \widetilde \CM_{(0,1)}((\dot \Sigma,\del \dot \Sigma),(M, \vec R);\CJ_\lambda) \to \vec R.
$$
The basic generic transversality is the following.

\begin{thm}[$0$-jet evaluation transversality, Theorem \ref{thm:0-jet}]\label{thm:0jet-intro}
Both evaluation maps $\Ev^+$ and $\Ev_\del$ are
submersions on the open subset consisting of somewhere injective elements of
$\CM^{\CJ_\lambda}((\dot \Sigma,\del \dot \Sigma),(M, \vec R))$.
\end{thm}

\subsection{Review of the contact triad connection}
\label{subsec:connection}

In this subsection, we give a brief review of the notion of canonical connection of contact triad
$(M,\lambda, J)$ that was introduced in \cite{oh-wang:connection}. This connection suits best
for our tensorial calculations performed in the study of various component of the analyses of the moduli space of
contact instanons which lead to various output equations that enable us to analyse their
$L^2$-adjoint problem related to the application of Hahn-Banach theorem entering in the
transversality analysis of the moduli spaces. (See Section \ref{sec:adjoint-eq}
for the relevant tensor calculations.)

\begin{thm}[Contact Triad Connection \cite{oh-wang:connection}]\label{thm:connection}
For every contact triad $(M,\lambda,J)$, there exists a unique affine connection $\nabla$, called the contact triad connection,
 satisfying the following properties:
\begin{enumerate}
\item The connection $\nabla$ is  metric with respect to the contact triad metric, i.e., $\nabla g=0$;
\item The torsion tensor $T$ of $\nabla$ satisfies $T(R_\lambda, \cdot)=0$;
\item The covariant derivatives satisfy $\nabla_{R_\lambda} R_\lambda = 0$, and $\nabla_Y R_\lambda\in \xi$ for any $Y\in \xi$;
\item The projection $\nabla^\pi := \pi \nabla|_\xi$ defines a Hermitian connection of the vector bundle
$\xi \to M$ with Hermitian structure $(d\lambda|_\xi, J)$;
\item The $\xi$-projection of the torsion $T$, denoted by $T^\pi: = \pi T$ satisfies the following property:
\be\label{eq:TJYYxi}
T^\pi(JY,Y) = 0
\ee
for all $Y$ tangent to $\xi$;
\item For $Y\in \xi$, we have the following
$$
\del^\nabla_Y R_\lambda:= \frac12(\nabla_Y R_\lambda- J\nabla_{JY} R_\lambda)=0.
$$
\end{enumerate}
\end{thm}
From this theorem, we see that the contact triad connection $\nabla$ canonically induces
a Hermitian connection $\nabla^\pi$ for the Hermitian vector bundle $(\xi, J, g_\xi)$,
and we call it the \emph{contact Hermitian connection}.

Moreover, the following fundamental properties of the contact triad connection was
proved in \cite{oh-wang:connection}, which will be useful to perform tensorial calculations later.

\begin{cor}\label{cor:connection}
Let $\nabla$ be the contact triad connection. Then
\begin{enumerate}
\item For any vector field $Y$ on $M$,
\be\label{eq:nablaYX}
\nabla_Y R_\lambda = \frac{1}{2}(\CL_{R_\lambda}J)JY;
\ee
\item $\lambda(T)=d\lambda$.
\end{enumerate}
\end{cor}

We refer readers to \cite{oh-wang:connection} for more discussion on the contact triad connection and its relation with other related canonical type connections.

\bigskip

\noindent{\bf Acknowledgement:}  We would like to thank the referees for all her/his
careful reading of manuscript and pointing out multitude of inconsistent notations, providing
helpful suggestions to improve readability of the paper. Their suggestions and questions
much improve the exposition of the paper and hence readability thereof.

\part{Generic mapping transversality under the perturbation of boundaries}

\section{Fredholm theory of (relative) contact instantons}
\label{sec:Fredholm-theory}

We start with setting-up the proper framework for the study of generic
nondegeneracy results for the Reeb orbits and chords.

\subsection{Set-up for the study of generic nondegeneracy of Reeb orbits and chords}
\label{subsec:setup-nondegeneracy}

We first introduce the following definition
\begin{defn} \label{defn:Moore-path2}Let $T \geq 0$ and consider a curve $\gamma:[0,1] \to M$ be a
smooth curve. We say $(\gamma,T)$ an \emph{iso-speed  Reeb trajectory} if the pair satisfies
$$
\dot \gamma(t) = T R_\lambda(\gamma(t)), \, \int \gamma^*\lambda = T
$$
for all $t \in [0,1]$. If $\gamma(1) = \gamma(0)$, we call $(\gamma,T)$ an
iso-speed closed Reeb orbit and $T$ the \emph{action} of $\gamma$.
\end{defn}

\begin{rem} We remark that this representation of a Moore path is different from that of
the one \eqref{eq:Moore-path} given in the introduction of the present paper. The relationship
is via the coordinate transformation
$$
(\gamma,T) \mapsto (T, \gamma((\cdot)/T))
$$
which transforms the pair $(\gamma,T)$ in Definition \ref{defn:Moore-path2}
 to the one in \eqref{eq:Moore-path} as long as $T \neq 0$. This way of representing a Reeb chord as a Moore path
 is useful for the study of transversality of Reeb chords in that now the domains of 
 the paths $\gamma$ are fixed to $[0,1]$.
 (See Appendix \ref{sec:nondegeneracy-chords}.)
\end{rem}

We start with the case of closed orbits.

\begin{defn} Let $(\gamma,T)$ be an iso-speed closed Reeb orbit in the sense as above.
When $|T| > 0$ is minimal among such that $\gamma(1) = \gamma(0)$ with $\int \gamma^*\lambda \neq 0$,
we call the pair $(\gamma,T)$ a \emph{simple} iso-speed closed Reed orbit.
\end{defn}

We consider the relative version thereof.

\begin{defn}[Iso-speed Reeb chords \cite{oh:entanglement1}]\label{defn:isospeed-chords}
 Let $(R_0,R_1)$ be a pair of Legendrian submanifolds of $(M,\xi)$ and $T\geq 0$.
For given contact form $\lambda$, we say a pair $(\gamma, T)$ with $\gamma:[0,1] \to M$
is an iso-speed Reeb chord from $R_0$ to $R_1$
if
$$
\dot \gamma(t) = T  R_\lambda(\gamma(t)), \quad \gamma(0) \in R_0, \, \gamma(1) \in R_1.
$$
We call  such a pair $(\gamma, T)$  \emph{positive} (resp. \emph{negative}) if $T \geq 0$
(resp. if $T < 0$).
\end{defn}

 We alert readers that the constant
curve is not a Reeb trajectory in the standard sense in that it does not satisfy the Reeb trajectory equation $\dot x = R_\lambda(x)$, while
it satisfies $\dot x = 0 = 0\cdot R_\lambda(x)$ which shows that any constant curve valued at
a point from $R_0 \cap R_1$ is a iso-speed Reeb chord with speed 0. When $T > 0$, the reparameterization $\gamma_T:[0,T] \to M$
$$
\gamma_T(t): =  \gamma(t/T)
$$
satisfies $\dot x = R_\lambda(x)$ with the period $T> 0$, i.e., satisfies $\gamma_T(0) = \gamma_T(T)$.

\begin{rem}  Note that when $R_0 = R_1$, we have `lots of iso-speed Reeb chords'
arising from the constant chords. We will show that this component of constant chords
is nondegenerate in the Bott-Morse sense.
This is important in our study of contact instanton Legendrian Floer homology we introduce
in \cite{oh:entanglement1,oh:contacton-gluing} and \cite{oh-yso:J1B},
especially in its calculation when $R_1$ is contact isotopic to $R_0$ and $C^1$-close thereto.
This is the main reason why our generic nondegeneracy includes the constant trajectories
defined over the fixed domain $[0,1]$ and emphasizes the \emph{iso-speed} formulation of
the Reeb chords given in Definition \ref{defn:isospeed-chords}.
\end{rem}

We now study the property of nondegeneracy of the pair $(\gamma, T)$ by formulating the notion
of nondegeneracy precisely including the case of constant trajectories, i.e., the case with $T = 0$.

Let $(\gamma, T)$ be a closed Reeb orbit of action $T$.
By definition, we can write $\gamma(T) = \phi^T_{R_\lambda}(\gamma(0))$
for the Reeb flow $\phi^T= \phi^T_{R_\lambda}$ of the Reeb vector field $R_\lambda$.
In particular $p = \gamma(0)$ is a fixed point of the diffeomorphism $\phi^T$.
Since $\CL_{R_\lambda}\lambda = 0$, $\phi^T$ is a contact diffeomorphism and so
induces an isomorphism
$$
\Psi_\gamma : = d\phi^T(p)|_{\xi_p}: \xi_p \to \xi_p
$$
which is the linearization restricted to $\xi_p$ of the Poincar\'e return map.

\begin{defn} Let $T> 0$. We say a $T$-closed Reeb orbit $(T,\lambda)$ is \emph{nondegenerate}
if $\Psi_\gamma:\xi_p \to \xi_p$ with $p = \gamma(0)$ has not eigenvalue 1.
\end{defn}
When $T=0$, it is well-known that the constant loop is nondegenerate in the Bott-Morse sense.

For $T > 0$, the following generic nondegeneracy result is well-known to the experts, at least
for the case of closed Reeb orbits.
(See \cite[Appendix A]{ABW} for its proof.)

\begin{thm}[Albers-Braam-Wendl]\label{thm:ABW}
There exists a residual subset
$$
\CC^{\text{\rm reg}}(M,\xi) \subset \CC(M,\xi)
$$
such that for any $\lambda \in \CC^{\text{\rm reg}}(M,\xi)$ all the
closed Reeb orbits $\lambda$ are nondegenerate if $T> 0$.
\end{thm}

The main purpose of Appendix \ref{sec:nondegeneracy-chords}
is to prove the following generic nondegeneracy result
for Reeb chords which extends the above nondegeneracy results to
the case of  Reeb chords and to the Bott-Morse situation of constant chords.

\begin{thm} \label{thm:Reeb-chords}
Let $(M,\xi)$ be a contact manifold. Let  $(R_0,R_1)$ be a pair of Legendrian submanifolds
where either $R_0 \cap R_1 = \emptyset$ or $R_0 = R_1$.
\begin{enumerate}
\item For a given pair $(R_0,R_1)$, there exists a residual subset
$$
\CC^{\text{\rm reg}}(\xi;R_0,R_1) \subset \CC(M,\xi)
$$
such that for any $\lambda \in \CC^{\text{\rm reg}}(\xi;R_0,R_1) $ all
Reeb chords from $R_0$ to $R_1$ are nondegenerate for $T > 0$ when $R_0 \cap R_1 = \emptyset$,
and Bott-Morse nondegenerate for $R_0 = R_1$ with $T = 0$.
\item For a given contact form $\lambda$, there exists a residual subset of pairs $(R_0,R_1)$
of Legendrian submanifolds such that  all Reeb chords from $R_0$ to $R_1$ are
nondegenerate for $T > 0$ and
Bott-Morse nondegenerate when $T = 0$.
\end{enumerate}
\end{thm}

\subsection{Asymptotic convergence and vanishing of asymptotic charge}

Next we recall from \cite{oh-wang:CR-map1} (resp. from \cite{oh-yso:index})
the asymptotic convergence result of contact instantons
of finite energy $E(w) = E^\pi(w) + E^\perp(w) < \infty$
for the closed string case (resp. with Legendrian boundary
condition of pair $(R_0,R_1)$ for the open string case)
near the punctures of a Riemann surface $\dot \Sigma$, respectively.
(We refer to \cite{oh:contacton,oh:entanglement1,oh-yso:J1B}
for the precise definition of total energy.)

Let $\dot\Sigma$ be a punctured Riemann surface with punctures
$$
\{p^+_i\}_{i=1, \cdots, l^+}\cup \{p^-_j\}_{j=1, \cdots, l^-}
$$
equipped with a metric $h$ with cylinder-like ends
(resp. \emph{strip-like ends} for the open string case)
outside a compact subset $K_\Sigma$.
Let $w: \dot \Sigma \to M$ be any such smooth map.

 Under the hypotheses of nondegeneracy $\lambda$ (resp. of the pair
 $(\lambda,(R_0,R_1)$ for the open string case)
 and of asymptotic convergence at the punctures, we can associate two
natural asymptotic invariants at each puncture defined as
\bea
T & := & \lim_{r \to \infty} \int_{\{r\}\times S^1}(w|_{\{r\}\times S^1})^*\lambda
\label{eq:TQ-T-intro}\\
Q & : = & \lim_{r \to \infty} \int_{\{r\}\times S^1}((w|_{\{r\}\times S^1 })^*\lambda\circ j)\label{eq:TQ-Q-intro}
\eea
at each puncture.
(Here we only look at positive punctures. The case of negative punctures is similar.)
As in \cite{oh-wang:CR-map1}, we call $T$ the \emph{asymptotic contact action}
and $Q$ the \emph{asymptotic contact charge} of the contact instanton $w$ at the given puncture.

The proof of the following subsequence convergence result
is given in \cite[Theorem 6.4]{oh-wang:CR-map1}.

\begin{thm}[Subsequence Convergence-closed strings,
Theorem 6.4 \cite{oh-wang:CR-map1}]
\label{thm:subsequence-intro}
Let $w:[0, \infty)\times S^1 \to M$ satisfy the contact instanton equations \eqref{eq:contacton-Legendrian-bdy-intro}
of finite energy.
Then for any sequence $s_k\to \infty$, there exists a subsequence, still denoted by $s_k$, and a
massless instanton $w_\infty(\tau,t)$ (i.e., $E^\pi(w_\infty) = 0$)
on the cylinder $\R \times S^1$  that satisfies the following:
\begin{enumerate}
\item $\delbar^\pi w_\infty = 0$ and
$$
\lim_{k\to \infty}w(s_k + \tau, t) = w_\infty(\tau,t)
$$
in the $C^l(K \times S^1, Q)$ sense for any $l$, where $K\subset [0,\infty)$ is an arbitrary compact set.
\item $w_\infty^*\lambda = -Q\, d\tau + T\, dt$
\end{enumerate}
\end{thm}

In general $Q  = 0$ does not necessarily hold for the closed string case.
When $Q \neq 0$ combined with $T = 0$ happens, we say $w$ has
the bad limit of \emph{appearance of spiraling instantons along the Reeb core}. It is also proven in \cite{oh:contacton} that If $Q = 0 = T$,
then the puncture is removable. When $Q = 0$, which is always the case when contact instanton
is exact such as those arising from the symplectization case,
$w_\tau$ converges to a Reeb orbit of period $|T|$ exponentially fast.

Now we make the corresponding statement for
the open string case proved in
\cite{oh:contacton-Legendrian-bdy}, \cite{oh-yso:index}.

\begin{thm}[Subsequence Convergence; the case of open strings]\label{thm:subsequence-open-intro}
Let $w:[0, \infty)\times [0,1]\to M$ satisfy the contact instanton equations \eqref{eq:contacton-Legendrian-bdy-intro}.
Then for any sequence $s_k\to \infty$, there exists a subsequence, still denoted by $s_k$, and a
massless instanton $w_\infty(\tau,t)$ (i.e., $E^\pi(w_\infty) = 0$)
on the cylinder $\R \times [0,1]$  such that
$$
\lim_{k\to \infty}w(s_k + \tau, t) = w_\infty(\tau,t)
$$
in the $C^l(K \times [0,1], Q)$ sense for any $l$, where $K\subset [0,\infty)$ is an arbitrary compact set.
Furthermore, $w_\infty$ has $Q = 0$ and the formula $w_\infty(\tau,t)= \gamma(T\, t)$  with
asymptotic action $T$, where $\gamma$ is some Reeb chord joining $R_0$ and $R_1$ of period $|T|$.
\end{thm}

\begin{cor}[Vanishing Charge] Assume the pair $(\lambda, \vec R)$ is nondegenerate.
Let $w$ be as above with finite energy. Suppose that $w(\tau,\cdot)$ converges as $\tau \to \infty$ in the
strip-like coordinate at a puncture $p \in \del \Sigma$  with associated Legendrian pair $(R,R')$.
 Then its  asymptotic charge $Q$ vanishes at $p$.
 \end{cor}

\subsection{Off-shell description of moduli spaces}
\label{subsec:off-shell}

For the exposition of this section, we adapt the one given in
\cite{oh:contacton} to the current context of bordered contact instantons by
incorporating the Legendrian boundary condition. In particular, we consider general bordered
compact surfaces of arbitrary genus to be consistent with that of \cite{oh:contacton} (for the closed case),
and order the marked points starting from $k = 1$, not from $k=0$ as in \eqref{eq:contacton-Legendrian-bdy-intro}.

We will be mainly interested in the two cases:
\begin{enumerate}
\item A generic nondegenerate case of $R_1, \cdots, R_k$ which in particular
are mutually disjoint,
\item The case where $R_1, \cdots, R_k = R$.
\end{enumerate}

We now choose a $\lambda$-adapted CR-almost complex structure $J$.
Let $(\Sigma, j)$ be a bordered compact Riemann surface, and let $\dot \Sigma$ be the
punctured Riemann surface with $\{z_1,\ldots, z_k \} \subset \del \Sigma$, we consider the moduli space
$$
\CM((\dot \Sigma,\del \dot \Sigma),(M,\vec R);J), \quad \vec R = (R_1,\cdots, R_k)
$$
 of finite energy maps $w: \dot \Sigma \to M$ satisfying the equation
\eqref{eq:contacton-Legendrian-bdy-intro}.

The  second case is transversal in the Bott-Morse sense both for the Reeb
chords and for the moduli space of contact instantons, which is
rather straightforward and easier to handle, and so omitted.

For the first case, all the asymptotic Reeb chords  are nonconstant and have nonzero
action $T \neq 0$. In particular, the relevant punctures $z_i$
are not removable. Therefore we have the decomposition of the finite energy moduli space
$$
\CM((\dot \Sigma,\del \dot \Sigma),(M,\vec R);J) =
\bigcup_{\vec \gamma \in \prod_{i=0}^{k-1}\frak{Reeb}(R_i,R_{i+1})}
\CM((\dot \Sigma,\del \dot \Sigma),(M,\vec R);J;\vec \gamma)
$$
by the asymptotic convergence result from \cite{oh:contacton-Legendrian-bdy}.
Depending on the choice of strip-like coordinates we divide the punctures
$$
\{z_1, \cdots, z_k\} \subset \del \Sigma
$$
into two subclasses
$$
p_1, \cdots, p_{s^+}, q_1, \cdots, q_{s^-} \in \del \Sigma
$$
as the positive and negative boundary punctures. We write $k = s^+ + s^-$.

Let $\gamma^+_i$ for $i =1, \cdots, s^+$ and $\gamma^-_j$ for $j = 1, \cdots, s^-$
be two given collections of Reeb chords at positive and negative punctures
respectively. Following the notations from \cite{behwz}, \cite{bourgeois}
(but applied to the Reeb chords instead of closed Reeb orbits),
we denote by $\underline \gamma$ and $\overline \gamma$ the corresponding
collections
\beastar
\underline \gamma & = & \{\gamma_1^+,\cdots, \gamma_{s^+}^+\} \\
\overline \gamma & = & \{\gamma_1^-,\cdots, \gamma_{s^-}^-\}.
\eeastar
For each $p_i$ (resp. $q_j$), we associate the strip-like
coordinates $(\tau,t) \in [0,\infty) \times [0,1]$ (resp. $(\tau,t) \in (-\infty,0] \times [0,1]$)
on the punctured disc $D_{e^{-2 \pi K_0}}(p_i) \setminus \{p_i\}$
(resp. on $D_{e^{-2 \pi K_0}}(q_i) \setminus \{q_i\}$) for some sufficiently large $K_0 > 0$.

\begin{defn}\label{defn:Banach-manifold} We define
\be\label{eq:offshell-space}
\CF((\dot \Sigma, \del \dot \Sigma),(M, \vec R);\underline \gamma,\overline \gamma)
=: \CF(M,\vec R;\underline \gamma, \overline \gamma)
\ee
to be the set of smooth maps satisfying the boundary condition
\be\label{eq:bdy-condition}
w(z) \in R_i \quad \text{ for } \, z \in \overline{z_{i-1}z_i} \subset \del \dot \Sigma
\ee
and the asymptotic condition
\be\label{eq:limatinfty}
\lim_{\tau \to \infty}w((\tau,t)_i) = \gamma^+_i(T_i(t+t_i)), \qquad
\lim_{\tau \to - \infty}w((\tau,t)_j) = \gamma_j^-(T_j(t-t_j))
\ee
for some $t_i, \, t_j \in [0,1]$, where
$$
T_i = \int_0^1 (\gamma^+_i)^*\lambda, \quad T_j = \int_0^1 ( \gamma^-_j)^*\lambda.
$$
Here $t_i,\, t_j$ depends on the given analytic coordinate and the parameterizations of
the Reeb chords.
\end{defn}
We will fix the domain complex structure $j$ and its associated K\"ahler metric $h$.
We regard the assignment
\be\label{eq:Upsilon}
\Upsilon: w \mapsto \left(\delbar^\pi w, d(w^*\lambda \circ j)\right), \quad
\Upsilon: = (\Upsilon_1,\Upsilon_2)
\ee
as a section of the (infinite dimensional) vector bundle:
We first formally linearize and define a linear map
\be\label{eq:DUpsilonw}
D\Upsilon(w): \Omega^0(w^*TM,(\del w)^*T\vec R) \to \Omega^{(0,1)}(w^*\xi) \oplus \Omega^2(\Sigma)
\ee
where we have the tangent space
$$
T_w \CF = \Omega^0(w^*TM,(\del w)^*T\vec R).
$$
For the simplicity of notation, we also introduce the following notation.
\begin{notation}[Codomain of $\Upsilon$]\label{nota:CD} For given fixed $\vec R$, we define
\be\label{eq:CDw}
\CC\CD_{(J,\vec R),(j,w)}: = \Omega_J^{(0,1)}(w^*\xi) \oplus \Omega^2(\Sigma)
= \CH^{\pi(0,1)}_w \oplus \Omega^2(\Sigma)
\ee
and
\beastar
\CC\CD_{(J,\vec R)} & = & \bigcup_{(j,w) \in \CF} \{(j,w)\} \times \CC\CD_{(J,\vec R),(j,w)}\\
\CC\CD_{\vec R} & = & \bigcup_{J,(j,w) \in \CF} \{J\} \times \CC\CD_{(J,\vec R)}.
\eeastar
Here $\CC\CD$ stands for `codomain'.
\end{notation}

Since we will not vary $\vec R$ in the present discussion, we omit $\vec R$
from notation and simply write $\CC\CD = \CC\CD_{\vec R}$.
We also denote by $\CC\CD^{\text{\rm univ}}$ as the further union
\be\label{eq:CD-univ}
\CC\CD^{\text{\rm univ}} = \bigcup_{(J,\vec R)}\{(J,\vec R)\} \times  \CC\CD_{(J,\vec R)}.
\ee
Let $k \geq 2$ and $p > 2$. We denote by
\be\label{eq:CWkp}
\CW^{k,p}: = \CW^{k,p}((\dot \Sigma, \del \dot \Sigma),(M, \vec R); \underline \gamma,\overline \gamma)
\ee
the completion of the off-shell function space \eqref{eq:offshell-space}.
It has the structure of a Banach manifold modeled by the Banach space given by the following

\begin{defn}[Tangent space $T_w\CW^{k,p}$]\label{defn:tangent-space} We define
$$
W^{k,p} (w^*TM, (\del w)^*T\vec R; \underline \gamma,\overline \gamma)
$$
to be the set of vector fields $Y = Y^\pi + \lambda(Y) R_\lambda$ along $w$ that satisfy
\be\label{eq:tangent-element-pi}
\begin{cases}
Y^\pi \in W^{k,p}\left((\dot\Sigma, \del \dot \Sigma), (w^*\xi, (\del w)^*T\vec R)\right), \\
Y^\pi(z)\in T\vec R\quad \text{for }\, z \in \del \dot \Sigma
\end{cases}
\ee
and
\be\label{eq:tangent-element-lambda}
\begin{cases}
\lambda(Y) \in W^{k,p}((\dot \Sigma, \del \dot \Sigma),(\R, \{0\})),\\
\lambda(Y)(z) = 0 \quad \text{for }\, z \in \del \dot \Sigma
\end{cases}
\ee
\end{defn}
Here we use the splitting
$$
TM = \xi \oplus \span_\R\{R_\lambda\}
$$
where $\span_\R\{R_\lambda\}: = \CL$ is a trivial line bundle and so
$$
\Gamma(w^*\CL) \cong C^\infty\left((\dot \Sigma, \del \dot \Sigma), (\R,\{0\})\right).
$$
The above Banach space is decomposed into the direct sum
\be\label{eq:tangentspace}
W^{k,p}((\dot\Sigma,\del \dot \Sigma),( w^*\xi, (\del w)^*T\vec R))
\bigoplus W^{k,p}((\dot \Sigma,\del \dot \Sigma), ( \R, \{0\})) \otimes R_\lambda :
\ee
by writing $Y = (Y^\pi, g R_\lambda)$ with a real-valued function $g = \lambda(Y(w))$ on $\dot \Sigma$.
Here we measure the various norms in terms of the triad metric of the triad $(M,\lambda,J)$.

Now for each given $J$ and $w \in \CW^{k,p}((\dot \Sigma,\del \dot \Sigma), (M, \vec R);\underline \gamma,\overline \gamma)$,
we consider the Banach space
$$
\Omega^{(0,1)}_{k-1,p}(w^*\xi;J): = W^{k-1,p}(\Lambda_J^{(0,1)}(w^*\xi))
$$
the $W^{k-1,p} $-completion of $\Omega^{(0,1)}(w^*\xi) = \Gamma(\Lambda^{(0,1)}(w^*\xi)$ and form the bundle
\be\label{eq:CH01}
\CH_{k-1,p}^{(0,1)}(M,J): = \bigcup_{w \in \CW^{k,p}} \Omega^{(0,1)}_{k-1,p}(w^*\xi;J)
\ee
over $\CW^{k,p}$.

\begin{defn}\label{defn:CHCM01} We associate the Banach space
\be\label{eq:CH01-w}
\CC\CD^{(0,1)}_{k-1,p}(M,\lambda;J)|_w: = \Omega^{(0,1)}_{k-1,p}(w^*\xi;J) \oplus \Omega^2_{k-2,p}(\dot \Sigma)
\ee
to each $w \in \CW^{k,p}$ and form the bundle
\beastar
\CC\CD^{(0,1)}_{k-1,p}(M,\lambda;J)& : = & \bigcup_{w \in \CW^{k,p}} \CC\CD^{(0,1)}_{k-1,p}(M,\lambda)|_w\\
& \cong & \CH_{k-1,p}^{(0,1)}(M,J) \bigoplus \left(\CW^{k,p} \times \Omega^2_{k-2,p}(\dot \Sigma)\right)
\eeastar
over the Banach manifold $\CW^{k,p}$ given in \eqref{eq:CWkp}.
\end{defn}

Then we can regard the assignment
$$
\Upsilon_1: w \mapsto \delbar^\pi w
$$
as a smooth section of the bundle $\CH_{k-1,p}^{(0,1)}(M,\lambda) \to \CW^{k,p}$. Furthermore
the assignment
$$
\Upsilon_2: w \mapsto d(w^*\lambda \circ j)
$$
defines a smooth section of the trivial bundle
$$
\Omega^2_{k-2,p}(\Sigma) \times \CW^{k,p} \to \CW^{k,p}.
$$
We summarize the above discussion into the following lemma.

\begin{lem}\label{lem:Upsilon} Consider the vector bundle
$$
\CC\CD_{k-1,p}(M, \vec R;J) \to \CW^{k,p}.
$$
The map $\Upsilon$ continuously extends to a continuous section still denoted by
$$
\Upsilon: \CW^{k,p} \to \CC\CD_{k-1,p} (M,\vec R;J).
$$
\end{lem}

With these preparations, the following is a consequence of the exponential estimates established
in \cite{oh:contacton-Legendrian-bdy}.
(See \cite{oh-wang:CR-map1} for the closed string case of vanishing charge.)

\begin{prop}\label{prop:exp-decay}
Assume $\lambda$ is nondegenerate.
Let $w:\dot \Sigma \to M$ be a contact instanton and let $w^*\lambda = a_1\, d\tau + a_2\, dt$.
Suppose
\bea
\lim_{\tau \to \infty} a_{1,i} = 0, &{}& \, \lim_{\tau \to \infty} a_{2,i} = T(p_i)\nonumber\\
\lim_{\tau \to -\infty} a_{1,j} = 0, &{}& \, \lim_{\tau \to -\infty} a_{2,j} = T(q_j)
\eea
at each puncture $p_i$ and $q_j$.
Then $w \in \CW^{k,p}(M, \vec R;\underline \gamma,\overline \gamma)$.
\end{prop}

Now we are ready to define the moduli space of contact instantons with prescribed
asymptotic condition.
\begin{defn}\label{defn:tilde-modulispace} Consider the zero set of the section $\Upsilon$
\be\label{eq:defn-tildeMM}
\widetilde \CM(M,\lambda, \vec R;\underline \gamma,\overline \gamma;J) =  \Upsilon^{-1}(0)
\ee
in the Banach manifold $\CW^{k,p}(M,\lambda,\vec R;\underline \gamma,\overline \gamma)$.
We write $w \sim w'$ for two elements therefrom if there is a biholomorphism
$\varphi$ of the punctured bordered Riemann surfaces $(\dot \Sigma, \del \dot \Sigma)$ such that
$w' = w \circ \varphi$, and define the quotient space and
\be\label{eq:defn-MM}
\CM(M,\lambda, \vec R;\underline \gamma,\overline \gamma;J)
= \widetilde \CM(M,\lambda, \vec R;\underline \gamma,\overline \gamma;J)/\sim
\ee
to be the set of equivalence classes of contact instantons $w$ under the equivalence relation $\sim$.
\end{defn}
This definition does not depend on the choice of $k$ or $p$  as long as $k\geq 2, \, p>2$.
We call an equivalence class $[w]$ an isomorphism class of contact instantons and often just write it as $w$
by an abuse of notation whose meaning should be clear from the give context.

\subsection{Linearized operator and its ellipticity}
\label{subsec:fredholm}

Let $(\dot \Sigma, j)$ be a punctured Riemann surface, the set of whose punctures
may be empty, i.e., $\dot \Sigma = \Sigma$ is either a closed or a punctured
Riemann surface. In this subsection and the next, we lay out the precise relevant off-shell framework
of functional analysis, and  establish the Fredholm property of the linearization map.
(We recall that we have already computed the linearization of $\Upsilon$ for the closed string case
in \cite{oh:contacton}.)

We recall that both for the elliptic regularity esimates in
\cite{oh:contacton-Legendrian-bdy,oh-yso:index} and for the optimal expression of the linearization map
and its relevant calculations in \cite{oh:contacton}, we have been using the contact triad connection $\nabla$ of $(M,\lambda,J)$ and the contact Hermitian connection $\nabla^\pi$ for $(\xi,J)$ introduced in
\cite{oh-wang:connection,oh-wang:CR-map1}. Likewise we will utilize the contact triad connection
in the following presentation of the linearization of contact instantons.

Then we have the following explicit formulae thereof.

\begin{thm}[Theorem 10.1 \cite{oh:contacton}; See also Theorem 1.15
\cite{oh-savelyev}] \label{thm:linearization} In terms of the decomposition
$dw = d^\pi w + w^*\lambda\, R_\lambda$
and $Y = Y^\pi + \lambda(Y) R_\lambda$, we have
\bea
D\Upsilon_1(w)(Y) & = & \delbar^{\nabla^\pi}Y^\pi + B^{(0,1)}(Y^\pi) +  T^{\pi,(0,1)}_{dw}(Y^\pi)\nonumber\\
&{}& \quad + \frac{1}{2}\lambda(Y) (\CL_{R_\lambda}J)J(\del^\pi w)
\label{eq:DUpsilon1}\\
D\Upsilon_2(w)(Y) & = &  - \Delta (\lambda(Y))\, dA + d((Y^\pi \rfloor d\lambda) \circ j)
\label{eq:DUpsilon2}
\eea
where $B^{(0,1)}$ and $T_{dw}^{\pi,(0,1)}$ are the $(0,1)$-components of $B$ and
$T_{dw}^\pi$, where $B, \, T_{dw}^\pi: \Omega^0(w^*TM) \to \Omega^1(w^*\xi)$ are
 zero-order differential operators given by
\be\label{eq:B}
B(Y) =
- \frac{1}{2}  w^*\lambda \otimes \left((\CL_{R_\lambda}J)J Y\right)
\ee
and
\be\label{eq:torsion-dw}
T_{dw}^\pi(Y) = \pi T(Y,dw)
\ee
respectively.
\end{thm}

From the above expression of the covariant linearization of
of the section $\Upsilon = (\Upsilon_1,\Upsilon_2)$, the linearization
continuously extends to a bounded linear map
$$
D\Upsilon_{(\lambda,T)}(w): T\CW^{k,p} \to \CC\CD_{k-1,p}(M,\lambda)
$$
where we recall
\beastar
 T\CW^{k,p} & = &
\Omega^0_{k,p}\left(w^*TM,(\del w)^*T\vec R \right) \\
\CC\CD_{k-1,p}(M,\lambda) & = & \Omega^{(0,1)}_{k-1,p}(w^*\xi;J) \oplus \Omega^2_{k-2,p}(\Sigma)
\eeastar
for any choice of $k \geq 2, \, p > 2$. Using the decomposition
$$
\Omega^0_{k,p}(w^*TM,(\del w)^*T\vec R) \cong \Omega^0_{k,p}(w^*\xi,(\del w)^*T\vec R) \oplus
\Omega^0_{k,p}(\dot \Sigma,\del \dot \Sigma)\otimes R_\lambda,
$$
$D\Upsilon(w)$ can be written into the matrix form
\be\label{eq:matrixDUpsilon}
\left(\begin{matrix}\delbar^{\nabla^\pi} + T_{dw}^{\pi,(0,1)} + B^{(0,1)}
 & d\left((\cdot) \rfloor d\lambda) \circ j\right) \\
\frac{1}{2} \lambda(\cdot) (\CL_{R_\lambda}J)J \del^\pi w & -\Delta(\lambda(\cdot)) \,dA
\end{matrix}
\right).
\ee
It follows that the map $D\Upsilon(w)$ is a partial differential operator whose principal
symbol map is given by $\sigma(D\Upsilon) = \sigma(D\Upsilon_1) \oplus \sigma(D\Upsilon_2)$
where
\bea\label{eq:symbol}
\sigma(D\Upsilon_1(w))(\eta) & = & J\Pi^*\eta \nonumber\\
\sigma(D\Upsilon_2(w))(\eta) & = & \langle \lambda,\eta\rangle^2 = (\eta(R_\lambda))^2
\eea
where $\eta$ is a cotangent vector in $T^*M \setminus \{0\}$ and
has decomposition
\be\label{eq:decompose-eta}
\eta = \eta^\pi + \eta(R_\lambda)\otimes w^*\lambda.
\ee
(See \cite{lockhart-mcowen} for the discussion of general elliptic operators of mixed degree
on noncompact manifolds with cylindrical ends.)

In particular we note that the restriction $D\Upsilon_1(w)|_{\Omega^0(w^*\xi)}$ has the same
principal symbol as that of
$$
\delbar^{\nabla^\pi} : \Omega^0(w^*\xi, (\del w)^*\xi) \to \Omega^{(0,1)}(w^*\xi;J)
$$
which is the first order elliptic operator of Cauchy-Riemann type, and that
$D\Upsilon_2(w)$ has the symbol of the Hodge Laplacian acting on zero forms
$$
*\Delta: \Omega^0(\dot \Sigma,\del \dot \Sigma) \to \Omega^2(\dot \Sigma).
$$

\subsection{Fredholm theory on punctured bordered Riemann surfaces}

By the (local) ellipticity shown in the previous subsection, it remains to examine the
Fredholm property of the linearized operator $D\Upsilon(w)$. For this purpose,
we need to examine the asymptotic behavior of the operator near punctures in strip-like
coordinates.

We first decompose the section $Y \in w^*TM$ into
$$
Y = Y^\pi + \lambda(Y) \otimes R_\lambda
$$
as before. Then the matrix \eqref{eq:matrixDUpsilon} has its entries given by
\bea
D\Upsilon_1^1(w)(Y^\pi) & = & \delbar^{\nabla^\pi}Y^\pi + B^{(0,1)}(Y^\pi) +  T^{\pi,(0,1)}_{dw}(Y^\pi), \label{eq:D}\\
D\Upsilon_2^1(w)(Y^\pi) & = & d((Y^\pi \rfloor d\lambda) \circ j),\label{eq:DUpsilon21}\\
D\Upsilon_1^2(w)(\lambda(Y) R_\lambda) & = & \frac12 \lambda(Y) \CL_{R_\lambda}J J \del^\pi w,
\label{eq:DUpsilon12}\\
D\Upsilon_2^2(w)(\lambda(Y) R_\lambda) & = & - \Delta(\lambda(Y))\, dA. \label{eq:DUpsilon22}
\eea
Noting that $Y^\pi$ and $\lambda(Y)$ are independent of each other, we write
$$
Y = Y^\pi + f R_\lambda, \quad f: = \lambda(Y)
$$
where $f: \dot \Sigma \to \R$ is an arbitrary function satisfying the boundary condition
$$
Y^\pi(\del \dot \Sigma) \subset T\vec R, \quad f|_{\del \dot \Sigma} = 0
$$
by the Legendrian boundary condition satisfied by $Y$. The following is obvious from
the expression of the $D\Upsilon_i^j(w)$.
\begin{lem} Suppose that $w$ is a solution to \eqref{eq:contacton-Legendrian-bdy-intro}.
The operators $D\Upsilon_i^j(w)$ have the following continuous extensions:
\beastar
D\Upsilon_1^1(w)(Y^\pi)& : & \Omega^0_{k,p}(w^*\xi,(\del w)^*T\vec R) \to \Omega^{(0,1)}_{k-1,p}(w^*\xi;J) \\
D\Upsilon_2^1(w)(Y^\pi) &: & \Omega^0_{k,p}(w^*\xi,(\del w)^*T\vec R)
 \to
\Omega^2_{k-1,p}(\dot \Sigma) \hookrightarrow \Omega^2_{k-2,p}(\dot \Sigma) \\
D\Upsilon_1^2(w)((\cdot) R_\lambda) & : & \Omega^0_{k,p}(\dot \Sigma,\del \dot \Sigma)
 \to
\Omega^{(0,1)}_{k,p}(w^*\xi;J) \hookrightarrow  \Omega^{(0,1)}_{k-1,p}(w^*\xi;J)\\
D\Upsilon_2^2(w)((\cdot) R_\lambda) & : & \Omega^0_{k,p}(\dot \Sigma,\del \dot \Sigma) \to \Omega^2_{k-2,p}(\Sigma).
\eeastar
\end{lem}

We regard the domains of  $D\Upsilon_i^2$ for $i=1, \,2$
as $C^\infty(\dot \Sigma, \del \dot \Sigma)$ using the isomorphism
\be\label{eq:Reeb-component-RR}
C^\infty(\dot \Sigma, \del \dot \Sigma) \cong \Omega^0(\dot \Sigma, \del \dot \Sigma) \otimes R_\lambda.
\ee
We now establish the following Fredholm property of the linearized operator.

\begin{prop}\label{prop:closed-fredholm}
Suppose that $w$ is a solution to \eqref{eq:contacton-Legendrian-bdy-intro}.
Consider the completion of $D\Upsilon(w)$,
which we still denote by $D\Upsilon(w)$, as a bounded linear map
from $\Omega^0_{k,p}(w^*TM,(\del w)^*T\vec R)$ to
$\Omega^{(0,1)}(w^*\xi)\oplus \Omega^2(\Sigma)$
for $k \geq 2$ and $p \geq 2$. Then
\begin{enumerate}
\item The off-diagonal terms of $D\Upsilon(w)$ are relatively compact operators
against the diagonal operator.
\item
The operator $D\Upsilon(w)$ is homotopic to the operator
\be\label{eq:diagonal}
\left(\begin{matrix}\delbar^{\nabla^\pi} + T_{dw}^{\pi,(0,1)}+ B^{(0,1)} & 0 \\
0 & -\Delta(\lambda(\cdot)) \,dA
\end{matrix}
\right)
\ee
via the homotopy
\be\label{eq:s-homotopy}
s \in [0,1] \mapsto \left(\begin{matrix}\delbar^{\nabla^\pi} + T_{dw}^{\pi,(0,1)} + B^{(0,1)}
& s\, d\left((\cdot) \rfloor d\lambda) \circ j\right)\\
\frac{s}{2} \lambda(\cdot) (\CL_{R_\lambda}J)J (\pi dw)^{(1,0)}
 & -\Delta(\lambda(\cdot)) \,dA
\end{matrix}
\right) =: L_s
\ee
which is a continuous family of Fredholm operators.
\item And the principal symbol
$$
\sigma(z,\eta): w^*TM|_z \to w^*\xi|_z \oplus \Lambda^2(T_z\Sigma), \quad 0 \neq \eta \in T^*_z\Sigma
$$
of \eqref{eq:diagonal} is given by the matrix
\beastar
\left(\begin{matrix} \frac{\eta + i\eta \circ j}{2} Id  & 0 \\
0 & |\eta|^2
\end{matrix}\right).
\eeastar
\end{enumerate}
\end{prop}
\begin{proof} Statement (1) is a consequence of the exponential decay near the puncture \cite{oh-yso:index},
 and the compactness of Sobolev embeddings
$$
\Omega^2_{k-1,p}(\dot \Sigma) \hookrightarrow \Omega^2_{k-2,p}(\dot \Sigma), \quad
\Omega^2_{k,p}(\dot \Sigma) \hookrightarrow \Omega^2_{k-2,p}(\dot \Sigma).
$$
When $\del \dot \Sigma = \emptyset$, the same kind of statement is proved in \cite{oh:contacton}.
Essentially the same proof applies by incorporating the boundary condition.
We leave some details and explanation on the requirement $d(w^*\lambda \circ j) = 0$
to \cite{oh:contacton-gluing}.
\end{proof}

Now we are ready to wrap-up the discussion of the Fredholm property of
the linearization map
$$
D\Upsilon_{(\lambda,T)}(w): \Omega^0_{k,p}(w^*TM, (\del w)^*T\vec R;\underline \gamma,\overline \gamma) \to
\Omega^{(0,1)}_{k-1,p}(w^*\xi;J) \oplus \Omega^2_{k-2,p}(\dot \Sigma)
$$
by proving Statement (1) of Proposition \ref{prop:closed-fredholm}.

The following proposition can be derived from the arguments used by
Lockhart and McOwen \cite{lockhart-mcowen} with the incorporation of
Legendrian boundary condition which is an elliptic boundary valued problem
as shown in \cite{oh:contacton-Legendrian-bdy}.

\begin{prop}[Proposition 11.6 \cite{oh:contacton}]\label{prop:fredholm} Assume that $\underline \gamma, \, \overline \gamma$ are
nondegenerate. Then the operator
\eqref{eq:matrixDUpsilon} is Fredholm.
\end{prop}

Then by the continuous invariance of the Fredholm index, we obtain
\be\label{eq:indexDXiw}
\operatorname{Index} D\Upsilon_{(\lambda,T)}(w) =
\operatorname{Index} \left(\delbar^{\nabla^\pi} + T^{\pi,(0,1)}_{dw}  + B^{(0,1)}\right) + \operatorname{Index}(-\Delta).
\ee
The computation of index is given in \cite{oh:contacton} for the closed string case
and is given for the current open string case in \cite{oh-yso:index}.

\begin{rem}\label{rem:weighted-setting}
Suppose $\delta > 0$ satisfies the inequality
$$
0\leq \delta < \min\left\{\frac{\text{\rm gap}(\vec \gamma)}{p}, \frac{\pi}{p}\right\}
$$
where $\text{\rm gap}(\vec \gamma)$ is the spectral gap,
\be\label{eq:gap}
\text{\rm gap}(\overline \gamma,\underline \gamma)
: = \min_{\gamma_i,\gamma_j}\left\{d_{\text{\rm H}}(\text{\rm Spec}A_{(T_i,\gamma_i)}, 0),
d_{\text{\rm H}}(\text{\rm Spec}A_{(T_j,\gamma_j)}, 0)\right\}
\ee
of the asymptotic operators $A_{(T_j,z_j)}$ or $A_{(T_i,z_i)}$
associated to the corresponding punctures. Then the above Fredholm property
also holds in the weighted Sobolev space setting with exponential weight $e^{\delta|\tau|}$.
\end{rem}

\section{Generic mapping transversality under the perturbation of $J$'s}
\label{sec:moduli-space}

In this section, we briefly recall the perturbation result under the perturbation of
CR almost complex structures from \cite{oh:contacton}
proved for the closed string case, and adapt it to the case of
boundary perturbations.

Let a contact manifold $(M,\xi)$ be given.
We consider the contact forms $\lambda$ of $(M,\xi)$ such that
all Reeb chords are nondegenerate. The set of such contact forms is
residual in $\CC(M,\xi)$. (See Appendix \ref{sec:nondegeneracy-chords} for the proof.)

Then we involve the set $\CJ_\lambda(M,\xi)$ of adapted $J$'s.
We study the linearization of the map $\Upsilon^{\text{\rm univ}}$ which is the map $\Upsilon$ augmented by
the argument $J \in \CJ_\lambda(M,\xi)$. More precisely, we define the universal section
$$
\Upsilon^{\text{\rm univ}}: \CF \times  \CJ_\lambda(M,\xi) \to
\CC\CD^{\text{\rm univ}}(M,\lambda)
$$
given by
\be\label{eq:Upsilon-univ}
\Upsilon^{\text{\rm univ}}((j, w), J) = \left(\delbar_J^\pi w, d(w^*\lambda \circ j)\right)
\ee
and study its linearization at each $(j,w,J) \in (\Upsilon^{\text{\rm univ}})^{-1}(0)$.
In the discussion below, we will fix the complex
structure $j$ on $\Sigma$, and so suppress $j$ from the argument of $\Upsilon^{\text{\rm univ}}$.

The following universal linearization formula plays a crucial role in the generic transversality
result as in the case of pseudoholomorphic curves in symplectic geometry.

\begin{lem}[Theorem 1.10 \cite{oh:contacton}]
\label{lem:DY-univ} Denote by $L: = \delta J$ the first variation of $J$.
We have the linearization
$$
D_{(w,J)} \Upsilon^{\text{\rm univ}}: T_w \CF \oplus T_J \CJ_\lambda(M,\xi) \to \Omega^{(0,1)}(w^*\xi) \bigoplus \Omega^2(\dot \Sigma)\otimes R_\lambda
$$
whose explicit formula is given by
$$
D_{(w,J)} \Upsilon^{\text{\rm univ}}(Y,L) = D_1 \Upsilon^{\text{\rm univ}}(Y)  + D_2 \Upsilon^{\text{\rm univ}}(L)
$$
where we have partial derivatives
\be\label{eq:D2}
D_1 \Upsilon^{\text{\rm univ}}(Y) = D\Upsilon(Y), \quad D_2 \Upsilon^{\text{\rm univ}}(L) = \frac12 L( d^\pi u \circ j)
\ee
\end{lem}
\begin{proof} This is straightforward from the definition
$$
\delbar^\pi w = \frac{d^\pi w + J d^\pi w \circ j}{2}
$$
and the fact that the projection $\pi$ does not depend on the choice of $J$ but depends
only on $\lambda$. We omit its proof.
\end{proof}

Following the procedure of considering the set $\CJ^\ell(M,\lambda)$ of $\lambda$-adapted
$C^\ell$ CR-almost complex structures $J$ inductively as $\ell$ grows (see \cite{mcduff-salamon-symplectic}, \cite[Section 10.4]{oh:book1}
for the detailed explanation),
we denote the zero set $(\Upsilon^{\text{\rm univ}})^{-1}(0)$ by
$$
\MM(M,\lambda,\vec R;\overline \gamma, \underline \gamma;\CJ_\lambda) = \left\{ (w,J)
\in \CW^{k,p}( M, \vec R;\overline \gamma, \underline \gamma) \times \JJ_\lambda^\ell(M,\xi)
\, \Big|\, \Upsilon^{\text{\rm univ}}(w, J) = 0 \right\}
$$
which we call the universal moduli space. Denote by
$$
\Pi_2: \CW^{k,p}( M, \vec R;\underline \gamma, \overline \gamma) \times \JJ_\lambda^\ell(M,\xi)
\to \JJ_\lambda^\ell(M,\xi)
$$
the projection. Then we have
\be\label{eq:MMK}
\MM(J;\underline \gamma, \overline \gamma)
= \MM(M,\lambda, \vec R; \underline \gamma, \overline \gamma;J)
 = \Pi_2^{-1}(J) \cap \MM(M,\lambda,  \vec R;\underline \gamma, \overline \gamma).
\ee
We state the following standard statement that often occurs in this kind of generic
transversality statement via the Sard-Smale theorem.

\begin{thm}\label{thm:trans} Let $0 < \ell < k -\frac{2}{p}$.
Consider the moduli space $\MM(M,\lambda;\underline \gamma, \overline \gamma)$. Then
\begin{enumerate}
\item $\MM(M,\lambda, \vec R;\underline \gamma, \overline \gamma)$ is
an infinite dimensional $C^\ell$ Banach manifold.
\item The projection
$$
\Pi_2|_{(\Upsilon^{\text{\rm univ}})^{-1}(0)} : (\Upsilon^{\text{\rm univ}})^{-1}(0) \to \JJ_\lambda^\ell(M,\xi)
\JJ^\ell(M,\lambda)
$$
is a Fredholm map and its index is the same as that of $D\Upsilon(w)$
for a (and so any) $w \in  \MM(M,\lambda, \vec R; J;\underline \gamma, \overline \gamma)$.
\end{enumerate}
\end{thm}

An immediate corollary of Sard-Smale theorem is that for a generic choice of $J$
$$
\Pi_2^{-1}(J) \cap (\Upsilon^{\text{\rm univ}})^{-1}(0)= \MM(J;\underline \gamma, \overline \gamma)
$$
is a smooth manifold: One essential ingredient for the generic transversality under the perturbation of
$J \in \CJ_\lambda(M,\xi)$ is the usage of the following unique continuation result.

\begin{prop}[Unique continuation lemma; Proposition 12.3 \cite{oh:contacton}]
\label{prop:unique-conti}
Any non-constant contact Cauchy-Riemann map does not
have an accumulation point in the zero set of $dw$.
\end{prop}

\begin{rem} The proof given in \cite{oh:contacton} utilizes the unique continuation result
through the (local) symplectization of contact instantons which become a (local) pseudoholomorphic
curves in the symplectization. We refer readers to the proof of Lemma \ref{lem:g=0}
for the unique continuation result  in the context of linearized problem whose proof is
given purely in terms of the analysis of contact instantons without taking the
symplectization. A similar proof can be also given to the nonlinear problem of
contact instantons by adapting the proof of Lemma \ref{lem:g=0}.
\end{rem}

\section{Generic transversality under the perturbation of boundaries}
\label{subsec:trans-under-bdy}

In this section, we study the problem of generic transversality under the boundary
Legendrian submanifolds imitating the arguments used in \cite{oh:fredholm} in the
framework of perturbations of Lagrangian submanifolds in symplectic geometry.

We put the following generic configuration of the Legendrian link $\vec R = (R_1, \cdots, R_n)$
for the study of generic transversality problems for the moduli space of bordered
contact instantons.

\begin{defn}[General position]\label{defn:general-position}
 We say that a Legendrian link $\vec R$ is in general position if the following hold:
\begin{enumerate}
\item Each pair $(R_i, R_j)$ is nondegenerate in the sense of Theorem \ref{thm:Reeb-chords}.
\item There is no triple $(R_i,R_j,R_k)$ for which no triple of Reeb chords that simultaneously overlaps on
an (relatively) open subsets of the images thereof.
\end{enumerate}
\end{defn}
This definition of general position being mentioned, we can perturb
the link $\vec R = \{R_1, \cdots, R_n\}$ componentwise
for the transversality studies under the perturbation of boundary conditions,
and may restrict ourselves to the case where $\vec R$ is a one-component link.

Similarly as in \cite[p.511]{oh:fredholm}, we fix a Legendrian submanifold $R_0$ and
represent each Legendrian submanifold $C^\infty$-close to the given $R_0$ as the one-jet graph
$$
\Image j^1f: = \{(x,df(x), f(x)) \in J^1R_0 \mid x \in R_0\}
$$
of a smooth function $f: R_0 \to \R$, contained  in a neighborhood of
the zero section $U_{R_0} \subset J^1R_0$ via the Darboux-Weinstein chart
$$
\Phi_{R_0}: V_{R_0} \subset M \to U_{R_0} \subset J^1R_0
$$
where $V_{R_0} \subset M$ is a neighborhood of $R_0$.
Then we can canonically represent each Legendrian submanifold $R$ $C^\infty$-close
to $R_0$ as $R = \phi_R(R_0)$ where
$$
\phi_R := \Phi_{R_0}^{-1} \circ j^1f: R_0 \to M
$$
where $f$ is the unique function on $R_0$ with $\Image j^1f= \Phi_{R_0}(R)$.
This map $\phi_R$ can be extended to an ambient contact isotopy $\psi_t: M \to M$
so that $\phi_R = \psi_1|_{R_0}$ so that  $\psi_1 \in \Cont_0(M,\xi)$ in particular.

We denote by $\CN(R_0) = \CN(R_0; \Phi_{R_0})$ the set of Legendrian submanifolds
given by
\be\label{eq:CNR0}
\CN(R_0): = \{ R \in \mathfrak{Leg}(M,\xi) \mid R = \Phi_{R_0}^{-1}(\Image j^1f), \,
f: R_0 \to \R, \, \Image j^1f \subset U_{R_0}\}.
\ee
Then we consider the subset
$$
\{(w,R) \mid R \subset U_{R_0},\, w(\del \dot \Sigma) \subset R \} \subset \CF \times \CN(R_0)
$$
and
\be\label{eq:moduli-near-R0}
\CM = (\CF  \times \CN(R_0))\cap \Upsilon^{-1}(0).
\ee

Following \cite{oh:fredholm}, we introduce the following notion.

\begin{defn}[Boundary somewhere injectivity] We say a map $w: \dot \Sigma \to M$ is
boundary somewhere injective if  there
exists a (open) subset $A \subset \del \dot \Sigma$ such that
$$
w^{-1}(w(z)) \cap \del \dot \Sigma = \{z\} \quad \text{\rm for $z \in A$}.
$$
\end{defn}

We consider the set of pairs
$$
(w,R) \in \CF((\dot \Sigma, \del \dot\Sigma), (M,R)) \times \mathfrak{Leg}(M,\xi)
$$
and the universal section map $\Upsilon = (\Upsilon_1, \Upsilon_2)$ of the bundle $\CC\CD \to \CF$
given by
$$
\Upsilon_1(w,R) = \delbar^\pi w, \quad \Upsilon_2(w,R) = d(w^*\lambda \circ j).
$$
Under the above general position hypothesis, we prove the following

\begin{thm}\label{thm:trans-under-bdy} Let $(M,\xi)$ be a contact manifold equipped with a
a contact form $\lambda$. We consider the same equation
considered in Theorem \ref{thm:trans}. Fix $J$ and $k$ and consider $\vec R$ in general position
given above. Then the subset
$$
\CM \subset \CF \times \CN(R_0)
$$
is a smooth $C^\infty$ submanifold (as a Frechet submanifold) near $w$ satisfying
the \emph{boundary somewhere injectivity}.
\end{thm}
\begin{proof} By the implicit function theorem, it is enough to prove
the universal section map $\Upsilon^{\text{\rm univ}}$ is a submersion at
$((w,J), \vec R)$ for any boundary somewhere injective $w$.

As usual in this kind of analysis, we consider the fiber bundle
$$
\CW^{k,p}(M, (\cdot);\underline \gamma, \overline \gamma ) \to {\mathcal Leg}(M,\xi)
$$
whose fiber at $(w,J, \vec{R})$ is given by
$$
\CW^{k,p}(M, \vec R;\underline \gamma, \overline \gamma).
$$
For the simplicity of exposition, we fix all $R_i$'s except $R_0$.
We consider a perturbation of $R_0$ under the contact isotopy of the type
$$
\psi(R_0), \quad \psi = \psi_H^1 \in \Cont_0(M,\xi)
$$
described as above.

Now we can modify the argument used in the proof of
the main theorem in \cite{oh:fredholm} as follows.
We denote by $\CL(M)$ and $\CL(R_0)$ the spaces of smooth Moore paths $(\gamma, T)$ on $M$ and $R_0$
respectively. (In \cite{oh:fredholm}, free loops are considered the set of which  are denoted by $\Omega(M)$ and $\Omega(L_0)$ respectively.)

We consider the parameterized smooth map
$$
\Upsilon^{\CL eg}: \CF \times \CN(R_0) \to \CC\CD \times \CL(M)
$$
by
$$
\Upsilon^{\CL eg}(w,R) : =\left (\Upsilon(w), \phi_R^{-1}\circ w|_{\del \dot \Sigma}\right).
$$
Then by definition we have
$$
\CM = (\Upsilon^{\CL eg})^{-1}(\{0\} \times \CL(R_0)).
$$
At this point, we can duplicate the rest of the proof of \cite[p.514-516]{oh:fredholm}
with obvious modification which is now in order.

To prove surjectivity, we will prove the $L^2$-cokernel vanishes, i.e.,
$$
\left(\Image D\Upsilon^{\CL eg}(w,R)
+ \{0\} \oplus T_{\phi_R^{-1} \circ w|_{\del \dot \Sigma}} \CL(R_0)\right)^\perp = 0
$$
as a subset of
$$
\Omega^{(1,0)}(w^*\xi) \oplus C^\infty(\dot \Sigma) \bigoplus T_{\phi_{R_0}^{-1} \circ w}\CL(M)
$$
after a suitable Sobolev completion: For the space $\Omega^{(1,0)}(w^*\xi)$,
we take the dual weighted Sobolev space $W_{-k,q;-\delta}$ for $W_{k,p;\delta}$
with $1/p + 1/q = 1$.
 (See \cite{lockhart-mcowen}, \cite[p.515]{oh:fredholm} for the detailed explanation
which however will not play much role in the calculation henceforth.)

Denote by $NR_0$ the normal bundle of $R_0$ (with respect to any given adapted metric $g$).
If
$$
((\eta,g), \alpha) \in \left(\Image D\Upsilon^{\CL eg}(w,R)
+ \{0\} \oplus T_{\phi_R^{-1} \circ w|_{\del \dot \Sigma}} \CL(R_0)\right)^\perp,
$$
it in particular implies that $\alpha$ satisfies
\be\label{eq:alphainNR0}
\alpha(z) \in N_{\phi_R\circ w(z)}R_0, \quad z \in \del \dot \Sigma.
\ee
(See \cite[p.516]{oh:fredholm} for the relevant calculations.)

Furthermore we denote by $\eta^{(1,0)}$ the value of the $(1,0)$-form $\eta$
evaluated against
$$
\frac{\del}{\del z}\cong
\frac12 \left(\frac{\del}{\del x} - j \frac{\del}{\del y}\right)
$$
\be\label{eq:eta10}
\eta^{(1,0)}: = \eta_x + J \eta_y
\ee
where $\eta = \eta_x dx + \eta_y dy$ in an isothermal coordinate $(x,y)$ of $\dot \Sigma$
adapted to $\del \dot \Sigma$.

By somewhat more complicated calculations involving
the integration by parts similarly as in \cite[p.516]{oh:fredholm},
we derive the following $L^2$-adjoint equation of associated linearized equation.
(See  the calculations given later in Section \ref{sec:proof}).)
We postpone its proof till the next section, Section \ref{sec:adjoint-eq}.

\begin{prop}\label{prop:L2-cokernel} Consider the pair
$$
((\eta,g),\alpha) \in \CC\CD \times T_{\phi_R^{-1} \circ w|_{\del \dot \Sigma}} \CL(M)
$$
and recall the correspondence
$$
(\eta,g) \longleftrightarrow \eta + g\, R_\lambda
$$
given by the identification \eqref{eq:Reeb-component-RR}.

Assume that $w$ is boundary somewhere injective and
 let $A \subset \del \dot \Sigma$ be a  nonempty (open) subset such that
$$
w^{-1}(w(z)) \cap \del \dot \Sigma = \{z\} \quad \text{\rm for $z \in A$}.
$$
Suppose that the pair $((\eta,g), \alpha)$ lies in the $L^2$-cokernel of $D\Upsilon^{\CL eg}(w,R)$.
Then the following hold and vice versa:
\begin{enumerate}
\item
$(Y,\alpha)$ satisfies
\be\label{eq:adjoint-eta}
\begin{cases}
(\delta^{\nabla^{\pi(0,1)} } + T^{\pi(1,0)} + B^{(1,0)}) (\eta^{(1,0)})
- J  \langle dg, dw \rangle  = 0 & \quad \text{\rm on }\,\dot \Sigma\\
(J \eta^{(1,0)} + g dw)\left(\frac{\del}{\del \nu}\right) -\alpha = 0
& \quad \text{\rm on }\, \del \dot \Sigma
\end{cases}
\ee
and
$$
\alpha^\perp = 0 \quad \text{\rm on }\, A.
$$
\item  The function $g$ satisfies
\be\label{eq:adjoint-g}
\begin{cases}
\Delta g -  \frac 12 g \left \langle (\CL_{R_\lambda}J)( J d^\pi w), \eta
\right \rangle  = 0, & \quad \text{\rm on }\, \dot \Sigma\\
g = 0 = \frac{\del g}{\del \nu}  & \quad \text{\rm on }\, \del \dot \Sigma.
\end{cases}
\ee
\end{enumerate}
\end{prop}

The equations \eqref{eq:adjoint-eta} and \eqref{eq:adjoint-g}
are nothing but the analog of the one right above \cite[Equation (3.13)]{oh:fredholm}.
In terms of the isothermal coordinates $(x,y)$ adapted to $\del \dot \Sigma$ and
in the usual convention for the calculus of vector valued differential forms
(see \cite{wells} or \cite[Appendix B]{oh-wang:CR-map1}), we can
express
\bea\label{eq:eta10}
\langle dg, dw \rangle &: =&  dg \wedge * dw
= \left(\frac{\del g}{\del x} \frac{\del w}{\del x} + \frac{\del g}{\del y} \frac{\del w}{\del y}\right)
\nonumber\\
J \eta^{(1,0)} \left(\frac{\del}{\del \nu}\right) & = & \eta_x - J\eta_y
\eea
for the one-form $\eta = \eta_x \, dx + \eta_y \, dy$.

Towards our goal of proving submersion property of the universal section map
$\Upsilon^{\text{\rm univ}}$, i.e., surjectivity of the linearization map
$\Upsilon^{\text{\rm univ}}_{(w,J),\vec R)}$ for boundary somewhere injective $w$,
we first derive the following from the equation \eqref{eq:adjoint-g}.

\begin{lem}\label{lem:g=0} $g = 0$.
\end{lem}
\begin{proof} We note that $g$ satisfies the following 2nd order linear elliptic equation
of the Laplcian type
\be\label{eq:Deltag+Ag}
\Delta g + A g = 0, \quad A = - \frac 12 \left \langle (\CL_{R_\lambda}J)( J d^\pi w), \eta
\right \rangle
\ee
where the coefficient function $A$ is smooth.
We apply the unique continuation as follows.

Let $z_0 \in \del \dot \Sigma$ be a point and $U \subset \dot \Sigma$ be its neighborhood.
Using the isothermal coordinate $(x,y)$ centered at $z_0$ we may identify $U$
with a semi-disc $D_\delta = \{(x,y) \mid |x| < \delta, \, y\geq 0\}$. Then the boundary
condition becomes
\be\label{eq:Cauchy-data}
g(x,0) = 0 = \frac{\del g}{\del y}(x,0)
\ee
for all $(x,0) \in \{y=0\} \cap D_\delta$. Then by Aronszajin's
unique continuation, more precisely \cite[Remark 2]{aronszajin} applied to
 \eqref{eq:Deltag+Ag} \emph{with the Cauchy data \eqref{eq:Cauchy-data}},
we conclude $g \equiv 0$. (Note that the Cauchy data together with the equation \eqref{eq:Deltag+Ag}
implies vanishing of infinity jets of $g$ along the boundary $U \cap \del \dot \Sigma$,
in particular at $z_0$. Since $z_0$ lies in the boundary of the domain $D_\delta$,
we apply \cite[Remark 2]{aronszajin}), not the commonly used version that 
usually applies to an interior point.)
\end{proof}

Then \eqref{eq:adjoint-eta} is reduced to
\be\label{eq:adjoint-eta-2}
\begin{cases}
\left(\delta^{\nabla^{\pi(0,1)} } + T^{\pi(1,0)} + B^{(1,0)}\right) (\eta^{(1,0)}) = 0
& \quad \text{\rm on }\,\dot \Sigma\\
-J \eta^{(1,0)} -\alpha = 0  & \quad \text{\rm on }\, \del \dot \Sigma,
\end{cases}
\ee
and
$$
\alpha^\perp = 0 \quad \text{\rm on }\, A.
$$
By combining \eqref{eq:alphainNR0} and the last
two boundary conditions, we derive
$$
\eta^{(1,0)}  = 0 \quad  \text{\rm on }\, A \subset \del \dot \Sigma.
$$
Then we derive $\eta = 0$ by the unique continuation from the first equation of
\eqref{eq:adjoint-eta}
$$
-\del^{\nabla^\pi} \eta + T_{dw}^{\pi, (1,0)} \eta + B^{(1,0)}\eta = 0
$$
which is of Cauchy-Riemann type. Then by back substitution of $\eta = 0$, we derive
$\alpha = 0$ again from $-J \eta^{(1,0)} - \alpha = 0$ on $\del \dot \Sigma$.
(See the proof of Proposition 3.4 in \cite[p.516]{oh:fredholm} for a similar argument
used.)

Finally we examine \eqref{eq:adjoint-g}.
Recall $\dot \Sigma$ has non-empty boundary punctures each of which has
strip-like coordinates $(\tau,t) \in \pm [0,\infty) \times [0,1]$. By the asymptotic convergence result of finite energy contact instantons $w$
and the given hypothesis, there exists a sufficiently large $R > 0$ such that
$dw|_{\del \dot \Sigma} \neq 0$ on $\dot \Sigma \cap \{(\tau,t) \mid t = 0, \, 1, \, |\tau| \geq R\}$
for a sufficiently large $R > 0$.  It is immediate to check the vanishing of $g$ by the
finite energy condition that such a harmonic function satisfying $g = 0 = \frac{\del g}{\del \nu}$
 must vanish which can be seen by expanding the finite energy harmonic function $g = g(\tau,t)$
 by a trigonometric series on the strip-like region.

Combining all the above, we have finished the proof of Theorem \ref{thm:trans-under-bdy}.
\end{proof}

An immediate corollary of Sard-Smale theorem then is the following generic
transversality result.

\begin{thm}
There exists a residual subset of $\vec R = (R_1, \ldots, R_k)$
of Legendrian submanifolds such that  the moduli space
$\MM(M,\lambda, \vec R;\underline \gamma, \overline \gamma;J) $ is transversal
so that it becomes a finite dimensional smooth manifold of given Fredholm index.
\end{thm}
It remains to prove that the $L^2$-cokernel element $((\eta,g),\alpha)$ is
characterized by the equations \eqref{eq:adjoint-eta} and \eqref{eq:adjoint-g}
which is now in order.

\section{Derivation of parametric $L^2$-adjoint equation}
\label{sec:adjoint-eq}

Our primary goal is to prove Proposition \ref{prop:L2-cokernel}, i.e., to show that
 whenever $((\eta,g), \alpha)$ lies in the $L^2$-cokernel of
 $D\Upsilon^{\CL eg}(w, \vec R)$ satisfies
$$
\langle D_w\Upsilon^{\CL eg}(w, \vec R)(Y, X_f), ((\eta,g), \alpha) \rangle = 0
$$
for all $Y$, then it satisfies \eqref{eq:adjoint-eta} and \eqref{eq:adjoint-g}
in Proposition \ref{prop:L2-cokernel}.

The entirety of the section will be occupied by the proof of this claim.
We first recall the correspondence \eqref{eq:Reeb-component-RR}.
Let $((\eta, g), \alpha) $ satisfy
$$
\left\langle \left(D\Upsilon^{\CD eg}(w,\vec R)\right)^\dagger((\eta,g),\alpha), (Y, X_f)
\right \rangle
= 0
$$
for all $Y \in T_w \CF$ and $f \in C^\infty(R_0)$. More explicitly, the equation is
\be\label{eq:int-L2-adjoint}
\int_{\dot \Sigma} \langle D_w\Upsilon^{\CL eg}(w,\vec R)(Y,X_f), (\eta,g) \rangle
+ \int_{\del \dot \Sigma} \langle \alpha, Y|_{\del \dot \Sigma}- X_f(w)|_{\del \dot \Sigma} \rangle = 0.
\ee
By definition, we have
$$
D_w\Upsilon^{\CL eg}(w,\vec R)(Y,X_f) = D_\Upsilon(w)(Y)
$$
where $D\Upsilon(w)$ is the linearization map \eqref{eq:DUpsilonw} of the section $\Upsilon$
defined by \eqref{eq:Upsilon} for the given Legendrian boundary condition $\vec R$.

Utilizing this, we  decompose the first integral of
\eqref{eq:int-L2-adjoint} into
\bea\label{eq:full-integral}
\int_{\dot \Sigma} \langle D_w\Upsilon^{\CL eg} (w,\vec R)(Y,X_f), (\eta,g) \rangle
& = & \int_{\dot \Sigma} \langle D\Upsilon_1(w)(Y), \eta \rangle
\nonumber\\
&{}& + \int_{\dot \Sigma} \langle D\Upsilon_2(w)(Y),\, g \, R_\lambda \rangle.
\eea
By the elliptic regularity, $(\eta,g)$ is smooth and so we can do the integration by parts.

For the calculation of the first integral, we recall
$$
D\Upsilon_1((w)(Y) =  \left(\delbar^{\nabla^\pi} + B^{(0,1)}
+ T^{\pi(0,1)}_{dw}\right)(Y^\pi)
+\frac12 \lambda(Y)  \left\langle (\CL_{R_\lambda}J)(J \del^\pi w), \eta \right \rangle.
$$
Here we recall the property of the linear map $\CL_{R_\lambda} J(\xi) \subset (\xi)$.

We take an isothermal coordinate $(x,y)$
near the boundary so that $h = dx^2 + dy^2$, $\del_x$ is tangent to $\del \dot \Sigma$ and $\del_y$ is
inward normal thereto. In particular, for the area element $dA$ on $\dot \Sigma$ and
the arc-length element $ds$ of $\dot \Sigma$ which are globally
defined on $\dot \Sigma$, we have
\be\label{eq:ds-dx}
dA = dx \wedge dy, \quad ds = dx
\ee
for any isothermal coordinate $(x,y)$ along the boundary $\del \dot \Sigma$ chosen as above.

Then we write
$\eta = \eta_x\, dx + \eta_y\, dy$ and
\beastar
\nabla^\pi Y^\pi & = & \nabla^\pi_{\del_x}Y^\pi \, dx +  \nabla^\pi_{\del_y}Y^\pi \, dy \\
(\nabla^\pi Y^\pi)\circ j & = & - \nabla^\pi_{\del_x}Y^\pi \, dy +  \nabla^\pi_{\del_y}Y^\pi \, dx.
\eeastar
Therefore we obtain
\beastar
\langle \delbar^{\nabla^\pi} Y^\pi, \eta\rangle \, dA & = &
\langle \delbar^{\nabla^\pi} Y^\pi, \eta\rangle dx \wedge dy \\
& = & d \left(- \langle Y^\pi, \eta_y - J\eta_x\rangle \, dx
+ \langle Y^\pi,\eta_x + J \eta_y \rangle dy\right) \\
& {}& + \langle Y^\pi, (- \nabla_{\del_x}^\pi + J\nabla^\pi_{\del_y}) (\eta_x + J \eta_y)  \rangle
\, dx\wedge dy.
\eeastar
We write
$$
(- \nabla_{\del_x}^\pi + J\nabla^\pi_{\del_y}) (\eta_x + J \eta_y)
=: \delta^{\nabla^{\pi(0,1)}}(\eta^{(1,0)}).
$$
Substituting these into the first integral of \eqref{eq:full-integral}, we obtain
\beastar
 \int_{\dot \Sigma} \langle D\Upsilon_1(w)(Y), (\eta,g) \rangle
 & = &
\int_{\dot \Sigma} \langle Y^\pi, (\delta^{\nabla^{\pi(0,1)} } + T^{\pi(1,0)}
+ B^{(1,0)}) (\eta^{(1,0)})
\rangle\\
&{}& + \int_{\del \dot \Sigma} - \langle Y^\pi, \eta_y - J\eta_x\rangle \, dx
 \\
 &{}& + \int_{\dot \Sigma} \frac12 g\, \lambda(Y)
 \left\langle (\CL_{R_\lambda}J)(J \del^\pi w), \eta \right \rangle.
\eeastar
Rearranging the terms around, we get
\bea
&{}& \int_{\dot \Sigma} \langle D\Upsilon_1((w,J),\vec R)(Y), (\eta,g)\rangle \nonumber\\
& = &
\int_{\dot \Sigma} \langle Y^\pi, (\delta^{\nabla^{\pi(0,1)} } + T^{\pi(1,0)} + B^{(1,0)}) (\eta^{(1,0)})
\rangle \nonumber\\
&{}& + \int_{\dot \Sigma}  \frac12 g\, \lambda(Y)
 \left\langle (\CL_{R_\lambda}J)(J \del^\pi w), \eta \right \rangle\, dA
\nonumber\\
&{}& + \int_{\del \dot \Sigma} - \langle Y^\pi, \eta_y - J\eta_x\rangle \, dx
\label{eq:first-integral}
\eea

Now we compute the second integral of \eqref{eq:full-integral}
$$
 \int_{\dot \Sigma} \langle D\Upsilon_2((w,J),\vec R)(Y)\, g \rangle
 = \int_{\dot \Sigma} g (-\Delta (\lambda(Y)))\, dA + g d((Y^\pi \rfloor d\lambda)\circ j).
$$
For the first, we apply Green's formula and the Legendrian boundary condition
which implies $\lambda(Y) = 0$ on $\del \dot \Sigma$, and get
\be\label{eq:Green's}
\int_{\dot \Sigma} g (-\Delta (\lambda(Y)))\, dA =
\int_{\dot \Sigma} -\Delta g\,  (\lambda(Y)))\, dA + \int_{\del \dot \Sigma}
- g  \frac{\del \lambda(Y)}{\del \nu} + \frac{\del g}{\del \nu} \lambda(Y).
\ee
For the second, we get
\beastar
 \int_{\dot \Sigma} g d((Y^\pi \rfloor d\lambda)\circ j) & = &
 -  \int_{\dot \Sigma} dg \wedge (Y^\pi \rfloor d\lambda)\circ j)
+ \int_{\del \dot \Sigma} g (Y^\pi \rfloor d\lambda)\circ j) \\
& = &  \int_{\dot \Sigma} dg \circ j \wedge (Y^\pi \rfloor d\lambda)
+ \int_{\del \dot \Sigma} g d\lambda \left(Y^\pi, \frac{\del w}{\del \nu}\right).
\eeastar
By summing the above two, we get
\bea\label{eq:second-integral}
&{}& \int_{\dot \Sigma} g (-\Delta (\lambda(Y)))\, dA + g d((Y^\pi \rfloor d\lambda)\circ j)
\nonumber\\
& = & \int_{\dot \Sigma} -\Delta g\,  (\lambda(Y)))\, dA + \int_{\del \dot \Sigma}
- g  \frac{\del \lambda(Y)}{\del \nu} + \frac{\del g}{\del \nu} \lambda(Y) \nonumber\\
&{}& + \int_{\dot \Sigma} dg \circ j \wedge (Y^\pi \rfloor d\lambda)
+ \int_{\del \dot \Sigma} g d\lambda \left(Y^\pi, \frac{\del w}{\del \nu}\right)
\nonumber \\
& = & \int_{\dot \Sigma} -\Delta g\,  (\lambda(Y)))\, dA
+  \int_{\dot \Sigma} d\lambda \left(Y^\pi, \frac{\del g}{\del x} \frac{\del w}{\del x}
+ \frac{\del g}{\del y} \frac{\del w}{\del y} \right) \nonumber\\
&{}& + \int_{\del \dot \Sigma} - g  \frac{\del \lambda(Y)}{\del \nu}+ \frac{\del g}{\del \nu} \lambda(Y)
+ \int_{\del \dot \Sigma} g d\lambda \left(Y^\pi, \frac{\del w}{\del \nu}\right)
\eea
By adding \eqref{eq:first-integral} and \eqref{eq:second-integral}, we obtain
\beastar
&{}& \int_{\dot \Sigma} \langle D\Upsilon((w,J),\vec R)(Y), (\eta,g) \rangle \nonumber\\
& = & \int_{\dot \Sigma} \langle Y^\pi, (\delta^{\nabla^{\pi(0,1)} } + T^{\pi(1,0)} + B^{(1,0)}) (\eta^{(1,0)})
\rangle \nonumber\\
&{}& + \int_{\dot \Sigma} -\Delta g\,  (\lambda(Y)))\, dA
+ \int_{\del \dot \Sigma} - g  \frac{\del \lambda(Y)}{\del \nu}+ \frac{\del g}{\del \nu} \lambda(Y) \nonumber\\
&{}& + \int_{\dot \Sigma} \frac12 g\, \lambda(Y)
 \left\langle (\CL_{R_\lambda}J)(J \del^\pi w), \eta \right \rangle \, dA
\nonumber\\
&{}& + \int_{\del \dot \Sigma} - \langle Y^\pi, \eta_y - J\eta_x\rangle \, dx
   \nonumber\\
 &{}& + \int_{\dot \Sigma} d\lambda \left(Y^\pi, \frac{\del g}{\del x} \frac{\del w}{\del x}
+ \frac{\del g}{\del y} \frac{\del w}{\del y} \right)
+ \int_{\del \dot \Sigma} g d\lambda \left(Y^\pi, \frac{\del w}{\del \nu}\right)\, ds.
\eeastar
 By rearranging terms, we get
 \bea\label{eq:1st-row-integral}
&{}& \int_{\dot \Sigma} \langle D\Upsilon((w,J),\vec R)(Y), (\eta,g) \rangle \nonumber\\
& = & \int_{\dot \Sigma} \langle Y^\pi, (\delta^{\nabla^{\pi(0,1)} } + T^{\pi(1,0)} + B^{(1,0)}) (\eta^{(1,0)})
\rangle \nonumber\\
 &{}& + \int_{\dot \Sigma} d\lambda \left(Y^\pi, \frac{\del g}{\del x} \frac{\del w}{\del x}
+ \frac{\del g}{\del y} \frac{\del w}{\del y} \right) \, dx \wedge dy \nonumber \\
&{}& + \int_{\dot \Sigma} -\Delta g\,  \,\lambda(Y)\, dA
 + \int_{\dot \Sigma}  \frac12 g\, \lambda(Y) \left\langle (\CL_{R_\lambda}J)(J \del^\pi w), \eta \right \rangle \, dA \nonumber \\
&{}& + \int_{\del \dot \Sigma} \left(- g  \frac{\del \lambda(Y)}{\del \nu} + \frac{\del g}{\del \nu} \lambda(Y)\right)\,
ds
\nonumber\\
&{}& + \int_{\del \dot \Sigma}  - \langle Y^\pi, \eta_y - J\eta_x\rangle \, dx
  + \int_{\del \dot \Sigma} g \, d\lambda \left(Y^\pi, \frac{\del w}{\del \nu}\right)\, ds.
\eea
Substituting this into \eqref{eq:int-L2-adjoint} followed by some rearrangement, we obtain
\beastar
0 & = & \int_{\dot \Sigma} \left\langle Y^\pi, (\delta^{\nabla^{\pi(0,1)} }
 + T^{\pi(1,0)} + B^{(1,0)}) (\eta^{(1,0)}) - J  \left(\frac{\del g}{\del x} \frac{\del w}{\del x}
+ \frac{\del g}{\del y} \frac{\del w}{\del y} \right)\right\rangle \, dA \\
&{}& +\int_{\dot \Sigma} \left(-\Delta g
 + \frac12 g\, \left\langle (\CL_{R_\lambda}J)(J\del^\pi w),\eta \right \rangle
 \right)\,  \lambda(Y)\ \, dA
\nonumber\\
&{}& + \int_{\del \dot \Sigma} \left(- g  \frac{\del \lambda(Y)}{\del \nu}+ \frac{\del g}{\del \nu} \lambda(Y) \right)\, ds\\
 &{}& + \int_{\del \dot \Sigma} - \left\langle Y^\pi, \eta_y - J\eta_x + g J\frac{\del w}{\del \nu}
-  \alpha \right \rangle \, ds \nonumber\\
&{}& - \int_{\del \dot \Sigma} \langle \alpha, X_f(w|_{\del \dot \Sigma})  \rangle\, ds
\eeastar
for all $(Y,f)$. From this and the fact that the choices of $\eta$ and of $\lambda(Y)$
are completely independent of each other,
we have derived the following equation
\be
\begin{cases}
(\delta^{\nabla^{\pi(0,1)} } + T^{\pi(1,0)} + B^{(1,0)}) (\eta^{(1,0)})
- J  \left(\frac{\del g}{\del x} \frac{\del w}{\del x}  + \frac{\del g}{\del y} \frac{\del w}{\del y} \right) = 0
& \quad \text{\rm on }\,\dot \Sigma\\
\eta_y - J\eta_x - g J\frac{\del w}{\del \nu} -\alpha= 0  & \quad \text{\rm on }\, \del \dot \Sigma
\end{cases}
\ee
and
\be\label{eq:g}
\begin{cases}
\Delta g -\frac12\mathcal{} g\, \left\langle (\CL_{R_\lambda}J)(J\del^\pi w) ,\eta \right\rangle
= 0 &  \quad \text{\rm on }\, \dot \Sigma\\
g = 0 = \frac{\del g}{\del \nu}  & \quad \text{\rm on }\, \del \dot \Sigma
\end{cases}
\ee
and
$$
\alpha^\perp = 0 \quad \text{\rm on }\, A.
$$
Here the vanishing $g = 0 = \frac{\del g}{\del \nu} $ follows from the integral
$$
 \int_{\del \dot \Sigma} \left(- g  \frac{\del \lambda(Y)}{\del \nu}+ \frac{\del g}{\del \nu} \lambda(Y)
 \right) \, dA
$$
since we can freely choose the function $f: = \lambda(Y)$ so that
the value of $\frac{\del f}{\del \nu}$ can be made arbitrary with the value of $f$ fixed
and vice versa.

Recalling the coordinate expression given in \eqref{eq:eta10},
we have finished the proof of Proposition \ref{prop:L2-cokernel}.

\part{Generic evaluation transversality}

\section{The evaluation transversality: statement}
\label{sec:ev-transverse-statement}

In this section, we start with the discussion on another important general ingredient
of the applications of contact instantons
to contact topology,  the evaluation map transversality,
 similarly as in the case of pseudoholomorphic curves.

We first recall the off-shell setting of the study of linearized operator in Theorem \ref{thm:linearization}:
Let $(M,\xi)$ be a contact manifold and consider contact triads $(M,\lambda,J)$ and
let $\vec R = (R_1, R_2, \ldots, R_k)$ be a Legendrian link.
We consider the associated contact instanton
equation
\be\label{eq:contacton-Legendrian-bdy}
\begin{cases}
\delbar^\pi w = 0, \, \quad d(w^*\lambda \circ j) = 0\\
w(\overline{z_iz_{i+1}}) \subset R_i, \quad i = 1, \ldots, k
\end{cases}
\ee
for a map $w:(\dot \Sigma, \del \dot \Sigma) \to (M,\vec R)$ with the boundary
condition given as above.

We consider the moduli space
$$
\CM(M,\lambda,\vec R;J):= \CM((\dot \Sigma,\del \dot \Sigma),(M,\vec R);J),
\quad \vec R = (R_1,\cdots, R_k)
$$
 of finite energy maps $w: \dot \Sigma \to M$ satisfying the equation
\eqref{eq:contacton-Legendrian-bdy}.

We also consider the space given in \eqref{eq:offshell-space}
$$
\CF(M,\lambda, \vec R;\underline \gamma,\overline \gamma)
$$
consisting of smooth maps satisfying the boundary condition \eqref{eq:bdy-condition}
and the asymptotic condition \eqref{eq:limatinfty}. Since we will not vary the pair $\underline \gamma,\overline \gamma$, we will often just write
$$
\CF = \CF(M,\lambda,\vec R) = \CF(M,\lambda, \vec R;\underline \gamma,\overline \gamma)
$$
and denote its marked version with $\ell$ interior and $k$ boundary marked points by
\be\label{eq:CFkell}
\CF_{(\ell,k)}(M,\lambda,\vec R).
\ee
We again consider the covariant linearized operator
$$
D\Upsilon(w): \Omega^0(w^*TM,(\del w)^*T\vec R) \to
\Omega^{(0,1)}(w^*\xi) \oplus \Omega^2(\Sigma)
$$
of the section
$$
\Upsilon: w \mapsto \left(\delbar^\pi w, d(w^*\lambda \circ j)\right), \quad
\Upsilon: = (\Upsilon_1,\Upsilon_2)
$$
as before.

We will treat the two cases,  evaluation at an interior marked point and one
at a boundary marked point,
separately. We denote by the subindex $(\ell,k)$ the number of interior and boundary marked points respectively.

Consider the parameterized marked moduli space
\beastar
&{}& \CM_{(1,0)}(M,\lambda, \vec R;\CJ_\lambda) \\
& = & \{((j,w),J, z) \mid w: \Sigma \to
M, \, \Upsilon(J,(j,w))  = 0,\, \, w(\del \dot \Sigma) \subset \vec R,\,  z \in \Int \dot \Sigma \}.
\eeastar
The evaluation map $\ev^+: \CM_{(1,0)}(M,\lambda, \vec R;J) \to M$ is defined by
$$
\ev^+((j,w),z) = w(z).
$$
We then have the fibration
$$
\widetilde \CM_{(1,0)}(M,\lambda, \vec R;\CJ_\lambda)  = \bigcup_{J \in \CJ_\lambda}
\widetilde \CM_{(1,0)} ((\dot \Sigma,\del \dot \Sigma),(M, \vec R);J)
\to \CJ_\lambda
$$
and
$$
\widetilde \CM_{(1,0)}^{\text{\rm inj}}(M, \lambda, \vec R;\CJ_\lambda)
$$
to be the open subset of
$\widetilde \CM_{(1,0)}(M,\lambda, \vec R;\CJ_\lambda)$ consisting of somewhere injective
contact instanton pairs $((j,w),J)$.
We have the universal ($0$-jet) evaluation map
$$
\Ev^+: \widetilde \CM_{(1,0)}(M,\lambda, \vec R;\CJ_\lambda) \to M.
$$
The basic generic transversality is the following.

\begin{thm}[$0$-jet evaluation transversality]\label{thm:0-jet}
The evaluation map
$$
\Ev^+: \widetilde \CM_{(1,0)}(M,\lambda, \vec R;\CJ_\lambda) \to M
$$
is a submersion. The same holds for the boundary evaluation map
$$
\Ev_\del: \widetilde \CM_{(0,1)}(M,\lambda, \vec R;\CJ_\lambda) \to \vec R.
$$
\end{thm}

\section{The interior evaluation transversality: proof}
\label{sec:ev-transverse-proof}

We closely follow the scheme exercised for the proof of
evaluation transversality given in \cite[Section 10.5]{oh:book1} which
in turn follows the scheme of the generic 1-jet transversality results proved
in \cite{oh-zhu:ajm}, \cite{oh:highjet} for the case of pseudoholomorphic curves in symplectic geometry.

An important ingredient in their proofs is the following
structure theorem  of the distributions with point support from
\cite[Section 4.5]{gelfand},  \cite[Theorem 6.25]{rudin} whose proof we refer readers thereto.
\begin{thm}[Distribution with point support]\label{thm:gelfand}
\index{distribution with point support}
Suppose $\psi$ is a distribution on open
subset $\Omega \subset \R^n$ with $\supp \psi \subset \{p\}$ and of finite
order $N < \infty$. Then $\psi$ has the form
$$
\psi = \sum_{|\alpha| \leq N} D^\alpha \delta_p
$$
where $\delta_p$ is the Dirac-delta function at $p$ and $\alpha =
(\alpha_1,\ldots, \alpha_n)$ is the multi-indices.
\end{thm}

We start with the case of interior marked point and consider the map
\bea\label{eq:aleph0}
\aleph_0 & : & \CJ_\lambda \times \CM_{\dot \Sigma}
\times \widetilde \CF_{(1,0)}(\dot \Sigma,M) \to \CC\CD
\times M \nonumber \\
&{}& (J,(j,w),z_0)  \mapsto  (\Upsilon^{\text{\rm univ}}(J,(j,w)),w(z_0)).
\eea
Here the subindex $0$ in $\aleph_0$ stands for the `0-jet', and the map $\Upsilon^{\text{\rm univ}}$
is the map given in \eqref{eq:Upsilon-univ}. (The higher-jet transversality can be also proved
by adapting the proof of \cite{oh:highjet} and the present proof. We postpone its proof and
an application elsewhere.)
Since the boundary condition is irrelevant for the discussion of the present section,
we omit $\vec R$ from the notation.

We denote by $\pi_i$ the projection from
$\CJ_\lambda \times \widetilde \CF_{(1,0)}(\dot \Sigma,M)$ to the $i$-th factor with $i=1, \, 2$.
Then we introduce
 \beastar
\widetilde \CM_{(1,0)}(\dot \Sigma,M;\{p\};\CJ_\lambda) & = & \aleph_0^{-1}(o_{\CC\CD} \times \{p\} )\\
\widetilde \CM_{(1,0)}(\dot \Sigma, M;\{p\};J)
& = & \widetilde \CM_{(1,0)}(\dot \Sigma, M;\{p\};\CJ_\lambda) \bigcap \pi_1^{-1}(J).
\eeastar

The following is a fundamental proposition for the proof of Theorem \ref{thm:0jet-intro}
as in the standard strategy exercised in the similar transversality result for the study of
pseudoholomorphic curves in  \cite[Section 10.5]{oh:book1} which in turn follows the scheme used in
\cite{oh-zhu:ajm}  for the 1-jet transversality proof for the case of pseudoholomorphic curves.
We apply the same scheme with the replacement of pseudoholomorphic curves by contact instantons
first for the 0-jet case in this part. Because the nature of equation is different, especially \emph{because the contact instanton
equation involves the second derivatives}, the proof involves additional complication beyond
that of \cite{oh-zhu:ajm}.

\begin{prop}\label{prop:0-jet}
The map $\aleph_0$ is transverse to the submanifold
$$
o_{\CC\CD} \times \{p\} \subset \CC\CD \times M.
$$
\end{prop}
\begin{proof} Its linearization $D\aleph_0(J,(j,w),z)$ is given by the map
\be\label{eq:DUpsilon}
(L,(b,Y),v) \mapsto
\left(D_{J,(j,w)}\Upsilon^{\text{\rm univ}}(L,(b,Y)), Y(w(z)) + dw(z)(v)\right) \ee
for
$$
L  \in T_J\CJ_\lambda, \, b \in T_j\CM_{\dot \Sigma}, \, v \in
T_z \dot \Sigma , \, Y \in T_w\FF(\Sigma, M).
$$
This defines a linear map
$$
T_J\CJ_\lambda \times T_j\CM_{\dot \Sigma} \times T_w\FF(\Sigma, M)
\times T_z \dot \Sigma  \to \CC\CD_{(J,(j,w))} \times T_{w(z)}M
$$
on $W^{1,p}$. But for the map $\aleph_0$ to be differentiable, we need to choose the
completion $W^{k,p}(\dot \Sigma,M)$ of $\FF(\dot \Sigma,M)$ with $k\geq 2$.

We take the Sobolev completion in the $W^{k,p}$-norm for at least $k \geq 2$. We take $k = 2$.
We would like to prove that this linear map is a submersion
at every element $(J, j, w, z_0) \in \widetilde\CM_1(\dot \Sigma,M)$ i.e.,
at the pair $(w,z_0)$ satisfying
$$
\Upsilon^{\text{\rm univ}}(J,(j,w)) = 0, \quad w(z_0) = p.
$$
For this purpose, we need to study solvability of the system of
equations
 \be D_{J,(j,w)}\Upsilon^{\text{\rm univ}}(L,(b,Y)) = (\gamma, \omega), \quad
Y(w(z_0)) + dw(v) = X_0
\ee
for any given $(\gamma, \omega) \in \CC\CD_w$ and $X_0$, i.e.,
$$
\gamma \in \Omega_{(j,J)}^{\pi(0,1)}(w^*TM), \,
\omega \in \Omega^2(\dot \Sigma), \quad X_0 \in T_{w(z_0)}M.
$$
For the current study of evaluation transversality,
the domain complex structure $j$ does not play much
role in our study. Especially it does not play any role throughout
our calculations except that it appears as a parameter. Therefore we will fix $j$
throughout the proof. Then it
will be enough to consider the case $b = 0 $. Then the above equation
is reduced to
\be\label{eq:b=0}
D_{J,w}\Upsilon^{\text{\rm univ}} (L,Y) = (\gamma,\omega), \quad Y(w(z_0)) + dw(v) = X_0.
\ee

Firstly,  we study \eqref{eq:b=0} for $Y \in W^{2,p}$. We regard
$$
\CC\CD_{(J,(j,w))} \times T_{w(z_0)}M
$$
as a Banach space with the norm $\|\cdot \|_{1,p} + \|\cdot \|_p +  |\cdot|$,
where $|\cdot|$ is any norm induced by an inner product on $T_xM$.

We will show that the image of the map
\eqref{eq:DUpsilon} restricted to the elements of the form
$$
(L,(0,Y),v)
$$
is onto as a map
$$
T_J \CJ_\lambda \times \Omega^0_{2,p}(w^*TM)
\to \CC\CD_{J,w}^{1,p} \times T_{w(z_0)}M
$$
where $(w,j,z_0,J)$ lies in $(\Upsilon_1^{\text{\rm univ}})^{-1}(o_{\CH^{''}}\times \xi)$, and we set
\be\label{eq:CD1p}
\CC\CD_{J,w}^{1,p} := \Omega^{(0,1)}_{1,p}(w^*\xi) \times \Omega^2_p (\dot \Sigma).
\ee
For the clarification of notations, we denote the natural pairing
$$
\CB \times \CB^* \to \R
$$
by $\langle \cdot, \cdot \rangle$ for any Banach space $\CB$ and the inner product on $T_xM$ by $(\cdot,
\cdot)_{x}$.

We now prove the following which will then finish the proof by the ellipticity
of the linearization map. The remaining part of proof will be occupied by the proof
of this statement.

\begin{prop}\label{prop:0jet-dense}
The subspace
$$
\Image \aleph_0 \subset \CC\CD_0 \oplus T_{w(z_0)}R
$$
is dense.
\end{prop}
\begin{proof} For the proof, we will use the Hahn-Banach lemma in an essential way.
Let $((\eta, f), X_p) \in \CC\CD^* \times T_pM$ satisfy
\be\label{eq:0}
\left\langle D_w\Upsilon(Y) + \left(\frac{1}{2}L \cdot d^\pi w \circ j,0\right),(\eta,f) \right\rangle
+ \langle Y, \delta_{z_0} X_p \rangle = 0
\ee
for all $Y \in \Omega^0_{2,p}(w^*TM)$ and $L$ where $\delta_{z_0}$ is the
Dirac-delta function supported at $z_0$.
By the Hahn-Banach lemma, it will be enough to prove
\be\label{eq:vanishing}
(\eta,f) = 0, \quad  X_p = 0.
\ee
In the derivation of \eqref{eq:0},  we have used the formula
$$
\Upsilon(J,(j,w)) = (\delbar^\pi_J w, d(w^*\circ \lambda \circ j))
$$
to compute the linearization of $\aleph_0$. In particular, we have
$$
D_1\aleph_0(L) = \left(\frac{1}{2}L \cdot d^\pi w \circ j,0\right)
$$
in the direction of $J$ where the second factor of the value of $\Upsilon$
does not depend on $J$.  Obviously, we have
$D_2\aleph_0(Y) = D\Upsilon_{J,j}(Y)$.

Under this assumption, we would like to show \eqref{eq:vanishing}.
Without loss of any generality, we
may assume that $Y$ is smooth since $C^\infty(w^*TM)
\hookrightarrow \Omega^0_{2,p}(w^*TM)$ is dense.
Taking $L=0$ in \eqref{eq:0}, we obtain
\be \label{eq:coker} %
\langle D_w\Upsilon(Y), (\eta,f) \rangle + \langle  Y,
\delta_{z_0} X_p \rangle = 0 \quad \mbox{ for all $Y$
of $C^{\infty}$ }.
\ee %
Therefore by definition of the distribution derivatives, $\eta$ satisfies
$$
(D_w\Upsilon)^\dagger (\eta,f) - \delta_{z_0} X_p = 0
$$
as a distribution, i.e.,
$$
(D_w\Upsilon)^\dagger (\eta,f) = \delta_{z_0} X_p
$$
where
$$
(D_w\Upsilon)^\dagger = (D_w \Upsilon(J,(j,w)))^\dagger
$$
is the formal adjoint of $D_w
\Upsilon(J,(j,w))$ whose symbol is of the same type as $D_w \Upsilon_{(j,J)}$ and so is
an elliptic first order differential operator. (See \eqref{eq:matrixDUpsilon} for the linearization formula and
recall that $(\delbar_J^\pi)^\dagger = - \del_J^\pi$ modulo zero order operators.)
By the elliptic regularity, $(\eta,f)$ is a classical solution on $\Sigma \setminus \{z_0\}$.

On the other hand, by setting $Y = 0$ in \eqref{eq:0},  we get
\be
\label{eq:1/2L} \langle L\cdot dw\circ j, \eta \rangle=0
\ee
for all $L\in T_{J} \CJ_\lambda$. From this identity, the
argument used in the transversality proven in the previous section
shows that $\eta=0$ in a small neighborhood of any somewhere injective point in $\Sigma
\setminus \{z_0\}$. Such a somewhere injective point exists by the
hypothesis of $w$ being somewhere injective and the fact that the
set of somewhere injective points is open and dense in the domain
under the given hypothesis. Then by the unique
continuation theorem, we conclude that $\eta = 0$ on $ \Sigma
\backslash\{z_0\}$ and so the support of $\eta$ as a distribution
on $\Sigma$ is contained at the one-point subset $\{z_0\}$ of $\Sigma$.

The following lemma will conclude the proof of Proposition \ref{prop:0jet-dense}.
We postpone the proof of the lemma till the next section.

\begin{lem} \label{lem:eta=0}  $(\eta,f)$ is a distributional solution of
$(D_w \Upsilon)^\dagger (\eta,f) = 0$ on $\Sigma$
and so continuous.
In particular, we have $(\eta,f) = 0$ in $(\CC\CD)^*$.
\end{lem}

Once we know $(\eta,f) = 0$, the equation \eqref{eq:0} is reduced to
the finite dimensional equation
\be\label{eq:simple0}
(Y(z_0),X_p)_{z_0} = 0
\ee
It remains to show that $X_p = 0$. For this, we have only to
show that the image of the evaluation map
$$
Y \mapsto  Y(z_0)
$$
is surjective onto $T_{p}M$, which is now obvious.

Now it remains to prove Lemma \ref{lem:eta=0}.

\section{Proof of Lemma \ref{lem:eta=0}.}
\label{sec:proof}

Our primary goal is to prove
\be\label{eq:tildexi} \langle D_w\Upsilon(Y), (\eta,f) \rangle = 0
\ee
for all smooth $Y \in \Omega^0(w^*TM)$, i.e., $(\eta,f)$ is a distributional solution of
$$
(D_w\Upsilon(J,(j,w)))^\dagger (\eta,f) = 0
$$
\emph{on the whole $\Sigma$}, not just on $\Sigma \setminus \{z_0\}$. This will imply that
$(\eta,f)$ is a solution smooth everywhere by the elliptic regularity.

We start with \eqref{eq:coker}
\be\label{eq:coker-append} \langle
D_w\Upsilon(Y), (\eta,f) \rangle + \langle Y,
\delta_{z_0}X_p \rangle = 0 \quad \mbox{for all $Y \in
C^{\infty}$}.
\ee
We first simplify the expression of the
pairing $\langle D_w\Upsilon(Y), (\eta,f) \rangle$ knowing that
$\supp (\eta,f) \subset \{z_0\}$.

Let $z$ be a complex coordinate centered at a fixed marked point $z_0$
and
$$
(x_1,y_1,x_2,y_2,\cdots,x_n,y_n, \eta)
$$
be a Darboux coordinates so that $\lambda = d\eta - \sum_{i=1}y_i dx_i$
on a neighborhood of $p \in M$.
\begin{rem} In the proof of \cite[Section 10.5]{oh:book1}, we chose a complex coordinates
$(w_1, \cdots, w_n)$ identifying a neighborhood of $p$ with an open subset of $\C^n$.
\end{rem}

We consider the standard metric
$$
h = \frac{\sqrt{-1}}{2} dz d\bar z
$$
on a neighborhood $U \subset \dot \Sigma$ of $z_0$.

The following lemma will be crucial in our proof.
\begin{lem}\label{lem:dJbar=dbar} Let $\eta$ be as above.
For any smooth section $Y$ of $w^*(TM)$ and $\eta$ of
$\left(\Omega^{(0,1)}_{1,p}(w^*\xi)\right)^*$
$$
\langle D\delbar_J^\pi(Y), \eta \rangle = \langle \delbar Y^\pi, \eta \rangle
$$
where $\delbar$ is the standard Cauchy-Riemann operators on $\R^{2n}\cong \C^n$ in the above coordinate.
\end{lem}
\begin{proof} We have already shown that
$\eta$ is a distribution with $\supp (\eta,f) \subset \{z_0\}$. By the
structure theorem on the distribution supported at a point $z_0$ Theorem \ref{thm:gelfand}, we have
$$
\eta = P\left(\frac{\del}{\del s}, \frac{\del}{\del t}\right)(\delta_{z_0})
$$
where $z = s + it$ is the given complex coordinates at $z_0$ and
$P\left(\frac{\del}{\del s}, \frac{\del}{\del t}\right)$ is a differential
operator associated by the polynomial $P$ of two variables with coefficients
in
$$
\Lambda^{(0,1)}_{(j_{z_0},J_p)}(w^*\xi).
$$

Furthermore since $\eta \in (W^{1,p})^*\cong W^{-1,q}$,
the degree of $P$ \emph{must be zero}
and so we obtain
\be\label{eq:eta=adelta}
\eta = \beta_{z_0}\cdot \delta_{z_0}
\ee
for some constant vector $\beta_{z_0} \in \Lambda^{(0,1)}(\xi_p) $.

We then have the expression
$$
D\delbar_J^\pi Y = \delbar Y + E \cdot \del Y + F\cdot Y
$$
near $z_0$ in coordinates  where $E$ and $F$ are zero-order matrix operators
satisfying
$$
E(z_0) = 0 = F(z_0).
$$
(See \cite[p.331]{oh-zhu:ajm} and \cite{sikorav:holo} for such a derivation.)
Therefore by \eqref{eq:eta=adelta}, we derive
\beastar
&{}& \langle E \cdot \del Y^\pi + F \cdot Y^\pi, \eta \rangle  =  \langle E \cdot \del Y^\pi
+ F\cdot Y^\pi, \beta_{z_0} \delta_{z_0} \rangle\\
&{}& \quad  = (E(z_0) \del Y^\pi(z_0) + F(z_0) Y^\pi(z_0), \beta_{z_0})_{z_0} = 0
\eeastar
and we obtain
$$
\langle D_w \delbar_J(Y), \eta \rangle = \langle \delbar Y^\pi + E \cdot \del^\pi Y^\pi + F\cdot Y^\pi,
\eta \rangle = \langle \delbar^\pi Y^\pi, \eta\rangle
$$
since $\supp \eta \subset \{z_0\}$.
This finishes the proof. \end{proof}

By this lemma, \eqref{eq:coker-append} becomes
\be\label{eq:coker-simple}
\langle \delbar Y^\pi, \eta \rangle + \langle -\Delta (\lambda(Y)) dA
+ d((Y^\pi \rfloor d\lambda) \circ j, f \rangle
+ \langle Y, \delta_{z_0} X_p \rangle = 0
\ee
for all $Y$. We next rewrite the middle summand by integration by parts
as in Part I but with the letter $g$ replaced by $f$ here.

\begin{lem}\label{lem:mid-summand} We have
$$
 \langle -\Delta (\lambda(Y)) dA
+ d((Y^\pi \rfloor d\lambda) \circ j, f \rangle
 = - \int \lambda(Y)\, \Delta f \, dA + \int df \circ j \wedge (Y^\pi \rfloor d\lambda)
 $$
 \end{lem}
 Recalling $Y = Y^\pi + \lambda(Y) \, R_\lambda$ and
noting that $\lambda(Y)$ and $Y^\pi$ are independent and arbitrary,
we rearrange the summand of
\eqref{eq:coker-simple} into
\beastar
 0 &= & - \int \lambda(Y)\, \Delta f \, dA  + \langle  \lambda(Y)\,  R_\lambda, \delta_{z_0} X_p \rangle \\
  &{}& + \int \langle  \delbar^\pi Y^\pi, \eta \rangle
  + \int df \circ j \wedge (Y^\pi \rfloor d\lambda) \\
  &{}& + \langle Y^\pi, \delta_{z_0} X_p \rangle
\eeastar
Recall the decomposition $Y = Y^\pi + \lambda(Y)\, R_\lambda$. By considering $Y$
with $Y^\pi = 0$ and with $\lambda(Y) = 0$ separately which are arbitrary,
we have derived
\bea
0 &= &- \int \lambda(Y)\, \Delta f \, dA +\langle  \lambda(Y)\, R_\lambda, \delta_{z_0} X_p \rangle
\label{eq:Reeb-term}\\
0 & = & \int \langle  \delbar Y^\pi, \eta \rangle
  + \int df \circ j \wedge (Y^\pi \rfloor d\lambda)
  + \langle Y^\pi , \delta_{z_0} X_p \rangle  \label{eq:Xi-term}.
 \eea
 We decompose $Y$ as
$$
Y(z) = (Y(z) - \chi(z) Y(z_0)) + \chi(z) Y(z_0)
$$
on $U$ where $\chi$ is a cut-off function with $\chi \equiv 1$ in a
small neighborhood $V \subset U$ of $z_0$ and satisfies $\supp \chi \subset
U$. It induces the corresponding decomposition of $Y^\pi$ and $\lambda(Y)$.

We first examine the equation \eqref{eq:Xi-term}.
Then the first summand $\widetilde Y^{\pi}$ defined by
$$
\widetilde Y^{\pi}(z):= Y^{\pi}(z) - \chi(z) Y^{\pi}(z_0)
$$
is a smooth section on $\Sigma$, and satisfies
$$
\widetilde Y^{\pi}(z_0) = 0, \quad \delbar \widetilde Y^{\pi} = \delbar{Y^{\pi}} \quad \mbox{on $V$}
$$
since $\chi(z) Y^{\pi}(z_0) \equiv Y^{\pi}(z_0)$ on $V$.
Therefore applying \eqref{eq:Xi-term} to $\widetilde Y^{\pi}$ instead of $Y^{\pi}$
and recalling $\supp (\eta,f) \subset \{p\}$, we obtain
$$
\langle \delbar \widetilde Y^{\pi}, \eta \rangle + \langle \widetilde Y^{\pi}, \delta_{z_0} X_p
\rangle = 0.
$$
Again using the support property $\supp \eta \subset \{z_0\}$
and \eqref{eq:coker-append},  we derive
\be\label{eq:delxi-adz}
\langle \widetilde Y^{\pi}, \delta_{z_0}X_p \rangle  =
\langle \widetilde Y^{\pi}(z_0), X_p\rangle= 0
\ee
and so $\langle \delbar \widetilde Y^{\pi}, \eta \rangle = 0$.
But we also have
\be\label{eq:tildexi=xi}
\langle \delbar Y^\pi, \eta \rangle=\langle \delbar \widetilde Y^\pi, \eta \rangle
\ee
since $\delbar^\pi \widetilde Y^\pi = \delbar{Y^\pi}$ on $V$ and $\supp \eta \subset \{z_0\}$.
Hence we obtain
$$
\langle \delbar^\pi_J Y^\pi, \eta \rangle = 0
$$
for all $Y^\pi$.

Applying similar reasoning to \eqref{eq:Reeb-term}, we have
derived
$$
\int \lambda(Y) \Delta f\, dA = 0
$$
for all $\lambda(Y)$. Combining the two, we have proved $(\eta,f)$ is a weak solution of
$$
(\delbar^\pi_J)^\dagger \eta = 0, \quad  \Delta f = 0
$$
on whole $\Sigma$. Therefore a we have finished the proof of \eqref{eq:tildexi} by Lemma \ref{lem:dJbar=dbar}.
By the elliptic regularity, $(\eta,f)$ is a smooth solution. In particular
it is continuous.  Since we have
already shown $(\eta,f) = 0$ on $\Sigma \setminus \{z_0\}$, continuity
of $\eta$ proves $(\eta,f) = 0$ on the whole $\Sigma$.
This finishes the proof.
 \end{proof}

This in turn finishes the proof of Proposition \ref{prop:0jet-dense}.
 \end{proof}

\section{The case of the boundary evaluation map}

In this section, we explain how we can augment the arguments used in the proof of
generic evaluation transversality to handle the case of boundary evaluation maps.
We will also write $R$ for $\vec R$ in the present section.

We now consider the map
\bea\label{eq:aleph1}
\aleph_0^\del & : & \CJ_\lambda \times \CM_{\dot \Sigma}
\times \widetilde \CF_{(0,1)}(\dot \Sigma, R) \to \CC\CD
\times R \nonumber \\
&{}& (J,(j,w),z_0)  \mapsto  (\Upsilon(J,(j,w)),w(z_0)).
\eea
Then for a given point $p \in R$,
 \beastar
\widetilde \CM_{(0,1)}(M, \lambda, R;\{p\};\CJ_\lambda) & = & \aleph_0^{-1}(o_{\CC\CD} \times \{p\} )\\
\widetilde \CM_{(1,0)}(M, \lambda, R;\{p\};J)
& = & \widetilde \CM_{(1,0)}( M, \lambda, R;\{p\};\CJ_\lambda) \cap \pi_1^{-1}(J).
\eeastar
We now establish the following boundary analog to Proposition \ref{prop:0-jet}.

\begin{prop}\label{prop:bdy-0jet}
The map $\aleph_0$ is transverse to the submanifold
$$
o_{\CC\CD} \times \{p\} \subset \CC\CD \times R.
$$
\end{prop}
\begin{proof} Its linearization $D\aleph_0(J,(j,w),z)$ is given by the map
\be\label{eq:DUpsilon-eval}
(L,(b,Y),v) \mapsto
\left(D_{J,(j,w)}\Upsilon(L,(b,Y)), Y(w(z)) + dw(z)(v)\right) \ee
for
$$
L  \in T_J\CJ_\lambda, \, b \in T_j\CM_{\dot \Sigma}, \, v \in
T_z \dot \Sigma , \, Y \in T_w\FF(M,\lambda,R)).
$$
But this time, $(L, (b,Y), v)$ satisfies the boundary condition
\be\label{eq:bdy-cond}
Y(\del \dot \Sigma) \subset T R, \quad v \in T \del \dot \Sigma.
\ee
This defines a linear map
$$
T_J\CJ_\lambda \times T_j\CM_{\dot \Sigma} \times T_w
\FF(M, \lambda,R)
\times T_z \dot \Sigma  \to \CC\CD_{(J,(j,w))} \times T_{w(z)} R.
$$

We take the Sobolev completion in the $W^{k,p}$-norm for with $k = 2$.
We would like to prove that this linear map is a submersion.
For this purpose, we again need to study solvability of the system of
equations
 \be D_{J,(j,w)}\Upsilon(L,(b,Y),v) = (\gamma, \omega), \quad
Y(w(z_0)) + dw(v) = X_0
\ee
for any given $(\gamma, \omega) \in \CC\CD_w$ and $X_0$, i.e.,
$$
\gamma \in \Omega_{(j,J)}^{\pi(0,1)}(w^*TM), \,
\omega \in \Omega^2(\dot \Sigma), \quad X_0 \in T_{w(z_0)} R.
$$
Again we put $b = 0 = v$ obtain the equation
\be\label{eq:bdy-b=0}
D_{J,w}\Upsilon (L,Y) = (\gamma,\omega), \quad Y(w(z_0)) = X_0
\ee
for $Y \in T_w \FF(M, \lambda,R)$.

We will show that the image of the map
\eqref{eq:DUpsilon-eval} restricted to the elements of the form
$$
(L,(0,Y),v)
$$
is onto as a map
$$
T_J \CJ_\lambda \times \Omega^0_{2,p}(w^*TM) \times T_z(\del \dot \Sigma)
\to \CC\CD_{J,w}^{1,p} \times T_{w(z_0)} R
$$
where $(w,j,z_0,J)$ lies in $\Upsilon_1^{-1}(o_{\CH^{''}}\times \xi)$.

We now prove the following boundary analog to Proposition \ref{prop:0jet-dense}
which will then finish the proof.

\begin{prop}\label{prop:bdy-0jet-dense}
The subspace
$$
\Image \aleph_0^\del \subset \CC\CD \oplus T_{w(z_0)}R
$$
is dense.
\end{prop}
\begin{proof}
Let $((\eta, f), X_p) \in (\CC\CD)^* \times T_p \vec R$ satisfy
\be\label{eq:0-bdy}
\left\langle D_w\Upsilon(Y) + \left(\frac{1}{2}L \cdot d^\pi w \circ j,0\right),(\eta,f) \right\rangle
+ \langle Y, \delta_{z_0} X_p \rangle = 0
\ee
for all $Y \in \Omega^0_{2,p}(w^*TM, (\del w)^*TR)$ and $L$ where $\delta_{z_0}$ is the
Dirac-delta function supported at $z_0$.
We again would like to show
\be\label{eq:bdy-vanishing}
(\eta,f) = 0, \quad X_p = 0.
\ee
Taking $L=0$ in \eqref{eq:0}, we obtain
\be \label{eq:bdy-coker}
\langle D_w\Upsilon(Y), (\eta,f) \rangle + \langle  Y,
\delta_{z_0} X_p \rangle = 0
\ee
for all $Y$ of $C^{\infty}$ satisfying the boundary condition
\be\label{eq:bdy-Y}
Y(\del \dot \Sigma) \subset T R.
\ee

Therefore by definition of the distribution derivatives, $\eta$ satisfies
$$
(D_w\Upsilon(J,(j,w)))^\dagger (\eta,f) - \delta_{z_0} X_p = 0
$$
as a distribution, i.e.,
$$
(D_w\Upsilon(J,(j,w)))^\dagger (\eta,f) = \delta_{z_0} X_p
$$
where \emph{$(D_w \Upsilon(J,(j,w)))^\dagger$ is the formal $L^2$-adjoint of $D_w
\Upsilon(J,(j,w))$.}
The following lemma provides a description of  the formal adjoint.

\begin{lem} The $L^2$-adjoint $(D_w \Upsilon(J,(j,w)))^\dagger$ is an linear elliptic
operator whose domain is given by the pairs $(\eta,f)$ such that
$$
\eta \in \Omega^{(1,0)}_{-1,q}(M,\lambda;J), \quad f \in L^q
$$
satisfying the elliptic boundary condition.
\end{lem}

By the similar reasoning by considering the  variations $L$ with $Y = 0$,
we again arrive at the following which will finish the proof by the same reason as
for the interior case.

\begin{lem} \label{lem:bdy-eta=0}  $\eta$ is a distributional solution of
$(D_w \Upsilon(J,(j,w)))^\dagger (\eta,f) = 0$ on $\Sigma$
and so continuous.
In particular, we have $(\eta,f) = 0$ in $(\CC\CD)^*$.
\end{lem}

Now it remains to prove Lemma \ref{lem:bdy-eta=0}. (In fact, we have only to establish
just near $\{z_0\}$ since the support of $(\eta,f)$ is concentrated at a point $z_0$.)
Our primary goal is to prove $(\eta,f)$ is a distributional solution of
$$
(D_w\Upsilon(J,(j,w)))^\dagger (\eta,f) = 0
$$
\emph{on an open set including $z_0$} (and so
on the whole $\Sigma$) by the same reason. The rest of the section will be occupied by
the proof of this goal.

We start with \eqref{eq:bdy-coker}
$$
\langle
D_w\Upsilon(Y), (\eta,f) \rangle + \langle Y,
\delta_{z_0}X_p \rangle = 0
$$
for all $Y \in C^{\infty}$ satisfying \eqref{eq:bdy-condition}, $Y(\del \dot \Sigma) \subset TR$.
Again knowing that
$$
\supp (\eta,f) \subset \{z_0\},
$$
we can  simplify the expression of the
pairing $\langle D_w\Upsilon(Y), (\eta,f) \rangle$ to
\be\label{eq:-bdy-coker-simple}
\langle \delbar^\pi Y^\pi, \eta \rangle + \langle -\Delta (\lambda(Y))\, dA
+ d((Y^\pi \rfloor d\lambda) \circ j, f \rangle
+ \langle Y, \delta_{z_0} X_p \rangle = 0
\ee
for all $Y$ satisfying $Y(\del \dot \Sigma) \subset T R$.

Now the boundary analog to Lemma \ref{lem:mid-summand} involves the boundary contribution which is
again by integration by parts combined with $\lambda(Y) \equiv 0$ on $\del \dot \Sigma$ by
the Legendrian boundary condition of $Y$.
(\emph{This is different from \eqref{eq:Green's} in that $\lambda(Y) = 0$ here while it
was arbitrary therein.})

\begin{lem}\label{lem:-bdy-mid-summand} We have
\beastar
&{}& \langle -\Delta (\lambda(Y)) dA
+ d((Y^\pi \rfloor d\lambda) \circ j, f \rangle\\
& = & - \int \lambda(Y)\, \Delta f \, dA + \int df \circ j \wedge (Y^\pi \rfloor d\lambda)
+ \int_{\del \dot \Sigma} f \frac{\del}{\del \nu}(\lambda(Y))\, d\theta  - f\, Y^\pi \rfloor d\lambda.
\eeastar
\end{lem}

Now we derive the boundary analogs to \eqref{eq:Reeb-term} and \eqref{eq:Xi-term} respectively:
\bea
0 &= &- \int \lambda(Y)\, \Delta f \, dA  + \int_{\del \dot \Sigma} f \frac{\del}{\del \nu}(\lambda(Y))\, d\theta \nonumber\\
&{}& +\langle  \lambda(Y)\, R_\lambda, \delta_{z_0} X_p \rangle
\label{eq:bdy-Reeb-term}\\
0 & = & \int \langle  \delbar^\pi Y^\pi, \eta \rangle
  + \int df \circ j \wedge (Y^\pi \rfloor d\lambda) -\int f\, Y^\pi \rfloor d\lambda \nonumber \\
&{}& \quad + \langle Y^\pi , \delta_{z_0} X_p \rangle
   \label{eq:bdy-Xi-term}.
 \eea
Again by replacing $Y$ by $\widetilde Y$ as before, we have now derived the following
boundary analogs to  \eqref{eq:Reeb-term} and \eqref{eq:Xi-term}
$$
\langle \delbar^\pi_J Y^\pi, \eta \rangle + \int df \circ j \wedge (Y^\pi \rfloor d\lambda) -\int f\, Y^\pi \rfloor d\lambda= 0
$$
and
$$
\int \lambda(Y) \Delta f\, dA + \int_{\del \dot \Sigma} f \frac{\del}{\del \nu}(\lambda(Y))\, d\theta = 0
$$
respectively for all $Y$ satisfying $Y(\del \dot \Sigma) \subset R$. Knowing that $Y^\pi$ and $\lambda(Y)$ are
independent functions, we have derived the equation
$$
\Delta f = 0, \quad f|_{\del \dot \Sigma} = 0.
$$
By the elliptic regularity, $f$ is continuous and hence $f \equiv 0$. (Recall $\supp f \subset \{p\}$.
Substituting this into the first, we get the equation for $\eta$ which satisfies
$$
(\delbar^\pi)^\dagger \eta = 0, \quad \eta|_{\del \dot \Sigma} \perp TR
$$
which is again an elliptic boundary value problem. Therefore $\eta$ is also smooth at $z_0$ and hence continuous.
Again we conclude $\eta \equiv 0$. Combining the two, we have finished the proof of Proposition \ref{prop:bdy-0jet-dense}.
\end{proof}

This in turn finishes the proof of Proposition \ref{prop:bdy-0jet}. \end{proof}

\appendix

\section{Generic nondegeneracy of Reeb chords}
\label{sec:nondegeneracy-chords}

Let $(M,\xi)$ be a contact manifold and $(R_0,R_1)$ be a pair of
Legendrian submanifolds.

We consider contact triads $(M,\lambda,J)$ and  consider the boundary value problem
for $(\gamma, T)$ with $\gamma:[0,1] \to M$
\be\label{eq:chord-equation}
\begin{cases}
\dot \gamma(t) = T R_\lambda(\gamma(t)),\\
\gamma(0) \in R_0, \quad \gamma(1) \in R_1.
\end{cases}
\ee
First we introduce the following nondegeneracy definition.
\begin{defn}\label{defn:nondegeneracy-chords} We say a Reeb chord $(\gamma, T)$ of $(R_0,R_1)$ is nondegenerate if
the linearization map $\Psi_\gamma = d\phi^T(p): \xi_p \to \xi_p$ satisfies
$$
\Psi_\gamma(T_{\gamma(0)} R_0) \pitchfork T_{\gamma(1)} R_1  \quad \text{\rm in }  \,  \xi_{\gamma(1)}.
$$
\end{defn}

\begin{rem} In \cite{oh:entanglement1}, the notion of \emph{Reeb trace} denoted by $Z_R$
of a Legendrian submanifold is introduced
$$
Z_R: = \bigcup_{t \in \R} \phi_{R_\lambda}^t(R)
$$
which is an immersed submanifold of dimension $\dim R + 1$. Then the above nondegneracy is
equivalent to the transversal intersection property
$$
\Psi_\gamma(T_{\gamma(0)} R_0) \pitchfork T_{\gamma(1)} Z_{R_1} \quad \text{\rm in}
\quad T_{\gamma(1)}M.
$$
\end{rem}

Similarly as in the problem of closed Reeb chords, we first consider the following
relative version of Reeb spectrum.

\begin{defn}\label{defn:spectrum} Let $\lambda$ be a contact form of contact manifold $(M,\xi)$ and $R \subset M$ a
connected Legendrian submanifold.
Denote by $\frak{Reeb}(M,\lambda)$ (resp. $\frak{Reeb}(M,R;\lambda)$) the set of closed Reeb chords
(resp. the set of self Reeb chords of $R$).
\begin{enumerate}
\item
We define $\operatorname{Spec}(M,\lambda)$ to be the set
$$
\operatorname{Spec}(M,\lambda) = \left\{\int_\gamma \lambda \mid \lambda \in \frak Reeb(M,\lambda)\right\}
$$
and call the \emph{action spectrum} of $(M,\lambda)$.
\item We define the \emph{period gap} to be the constant given by
$$
T(M,\lambda): = \inf\left\{\int_\gamma \lambda \mid \lambda \in \frak Reeb(M,\lambda)\right\} > 0.
$$
\end{enumerate}
We define $\operatorname{Spec}(M,R;\lambda)$ and the associated $T(M,\lambda;R)$ similarly using the set
$\frak{Reeb}(M,R;\lambda)$ of Reeb chords of $R$.
\end{defn}
We set $T(M,\lambda) = \infty$ (resp. $T(M,\lambda;R) = \infty$) if there is no closed Reeb orbit (resp. no $(R_0,R_1)$-Reeb chord). Then we define
\be\label{eq:TMR}
T_\lambda(M;R): = \min\{T(M,\lambda), T(M,\lambda;R)\}
\ee
and call it the \emph{(chord) period gap} of $R$ in $M$.

We denote by
$$
{\mathcal Leg}(M,\xi)
$$
the set of Legendrian submanifold and by ${\mathcal Leg}(M,\xi;R)$ its connected component
containing $R \in {\mathcal Leg}(M,\xi)$, i.e, the set of Legendrian submanifolds Legendrian isotopic to
$R$. We denote by
$$
\CP({\mathcal Leg}(M,\xi))
$$
the monoid of Legendrian isotopies $[0,1] \to {\mathcal Leg}(M,\xi)$. We have
natural evaluation maps
$$
\ev_0, \, \ev_1:\CP({\mathcal Leg}(M,\xi)) \to {\mathcal Leg}(M,\xi)
$$
and denote by
$$
\CP({\mathcal Leg}(M,\xi), R) = \ev_0^{-1}(R) \subset \CP({\mathcal Leg}(M,\xi))
$$
and
$$
\CP({\mathcal Leg}(M,\xi), (R_0,R_1)) = (\ev_0\times \ev_1)^{-1}(R_0,R_1) \subset \CP({\mathcal Leg}(M,\xi)).
$$
Finally we vary $\lambda$ $\mathfrak{Reeb}(M;\lambda)$
(resp. $\mathfrak{Reeb}(M,R;\lambda)$) for the given $(M,\xi)$ (resp. $((M,R),\xi)$) and form the
union
\be\label{eq:Reeb-xi}
\mathfrak{Reeb}(M,\xi) = \bigcup_{\lambda \in \mathfrak{Cont}(M,\xi)} \mathfrak{Reeb}(M;\lambda)
\ee
and
\be\label{eq:Reeb-Rxi}
\mathfrak{Reeb}(M,R, \xi) = \bigcup_{\lambda \in \mathfrak{Cont}(M,R, \xi)}
\mathfrak{Reeb}(M, R, \lambda).
\ee
\subsection{Under the perturbation of contact forms}
\label{subsec:Reeb-chords-lambda}

In this subsection, we prove the following relative version of Theorem \ref{thm:ABW}.

\begin{thm} \label{thm:Reeb-chord-lambda}
Let $(M,\xi)$ be a contact manifold. Let  $(R_0,R_1)$ be a pair of Legendrian submanifolds
allowing the case $R_0 = R_1$.  There
exists a residual subset $\operatorname{Cont}^{\text{\rm reg}}_1(M,\xi) \subset \CC(M,\xi)$
such that for any $\lambda \in \operatorname{Cont}^{\text{\rm reg}}_1(M,\xi)$ all
Reeb chords from $R_0$ to $R_1$ are nondegenerate for $T > 0$ and
Bott-Morse nondegenerate when $T = 0$.
\end{thm}

The case $R_0 = R_1$ with $T = 0$ is easy to prove which we omit referring its details
to \cite{oh:contacton-gluing}.  This being mentioned, we will focus on the case
$R_0 \cap R_1 = \emptyset$ in the following discussion.

Denote by $\LL(M;R_0,R_1)$ the space of paths
$$
\gamma: ([0,1], \{0,1\}) \to (M;R_0,R_1).
$$
We consider the assignment
\be\label{eq:Phi-TR}
\Phi: (T,\gamma,\lambda) \mapsto \dot \gamma - T \,R_\lambda(\gamma)
\ee
as a section of the Banach vector bundle over
$$
(0,\infty) \times \CL^{1,2}(M;R_0,R_1) \times \Cont(M,\xi)
$$
where $\CL^{1,2}(M;R_0,R_1)$
is the $W^{1,2}$-completion of $\CL(M;R_0,R_1)$. We have
$$
\dot \gamma - T\, R_\lambda(\gamma) \in \Gamma(\gamma^*TM; T_{\gamma(0)}R_0, T_{\gamma(1)}R_1).
$$
We  define the vector bundle
$$
\CL^2(R_0,R_1) \to (0,\infty) \times \CL^{1,2}(M;R_0,R_1) \times \Cont(M,\xi)
$$
whose fiber at $(T,\gamma,\lambda)$ is $L^2(\gamma^*TM)$. We denote by
$\pi_i$, $i=1,\, 2, \, 3$ the corresponding projections as before.

We denote $\frak{Reeb}(M,\lambda;R_0,R_1) = \Phi_\lambda^{-1}(0)$,
where
$$
\Phi_\lambda: = \Phi|_{ (0,\infty) \times \CL^{1,2}(M;R_0,R_1) \times \{\lambda\}}.
$$
Then by definition \eqref{eq:Reeb-xi}, we have
$$
\frak{Reeb}(\lambda;R_0,R_1) =  \Phi_\lambda^{-1}(0) = \frak{Reeb}(M,\xi) \cap \pi_3^{-1}(\lambda).
$$
\begin{prop} \label{prop:nondegeneracy-chords} Suppose $R_0 \cap R_1 = \emptyset$.
A Reeb chord $(\gamma, T)$ of $(R_0,R_1)$ is nondegenerate if and only if
the linearization
$$
d_{(\gamma, T)}\Phi: \R \times W^{1,2}(\gamma^*TM;T_{\gamma(0)}R_0,T_{\gamma(1)}R_1)
\to L^2(\gamma^*TM)
$$
is surjective.
\end{prop}
\begin{proof} We have the formula for $d_{(\gamma, T)}\Phi$ from \eqref{eq:Phi-TR}
$$
d_{(\gamma, T)}\Phi(a,\xi) = \frac{D\xi}{dt} - TDR_\lambda(\gamma)(\xi) - a R_\lambda(\gamma)
$$
acting on $\xi$ satisfying the boundary condition
$$
\xi(0) \in T_{\gamma(0)} R_0, \quad \xi(1) \in T_{\gamma(1)}R_1.
$$
Then  by the Fredholm alternative, we derive the $L^2$-cokernel of the operator
$d_{(\gamma, T)}\Phi$ is given by
\bea
\ker (d_{(\gamma, T)}\Phi)^\dagger &= & \Big \{ \eta \in \Gamma(\gamma^*TM) \mid  \nonumber \\
&{}& \quad \begin{cases}
\frac{D\eta}{d t}  + T\, DR_\lambda(\gamma)^\dagger \eta = 0, \label{eq:cokernel-1}
\\
\eta(0) \in N_{\gamma(0)}R_0, \quad \eta(1) \in N_{\gamma(1)}R_1
\end{cases}\\
&{}& \quad \int_0^1 \langle a R_\lambda(\gamma(t)),\eta(t)\rangle \, dt = 0
\forall a \in \R \Big\}. \label{eq:cokernel-2}
\eea

We first derive the following lemma.
\begin{lem} We have
$$
DR_\lambda(\gamma)^\dagger = J(\gamma) DR_\lambda(\gamma) J(\gamma)
$$
pointwise, where $J \in \operatorname{End}(TM)$ is given by $J = J_\xi \oplus id$
with respect to the splitting $TM = \xi \oplus \span\{R_\lambda\}$.
\end{lem}
\begin{proof} Since $\phi^t$ preserves $\lambda$, we obtain
$$
(\phi^t)^*d\lambda = d\lambda
$$
i.e., we have
$$
d\lambda(\phi^t(p))(d\phi^t(v_1),d\phi^t(v_2)) = d\lambda(p)(v_1,v_2).
$$
Regard $t \mapsto d_x\phi^t$ as a section of
$\operatorname{Hom}(T_xM, T_{\phi^{(\cdot)}(x)}M)\to \R$
of the vector bundle over $\R$. Then by
taking the covariant derivative with respect to the connection $\nabla$ preserving
$J$ and $d\lambda$ and utilizing the identity
$$
d(\phi^t)^{-1}  \frac{D}{dt}d\phi^t + \frac{D}{dt}d\phi^t d(\phi^t)^{-1}  = 0,
$$
we obtain
$$
d\lambda\left(\frac{D}{dt}d\phi^t (v_1), d\phi^t(v_2)\right) + d\lambda\left(d\phi^t (v_1), \frac{D}{dt}d\phi^t(v_2)\right) = 0.
$$
for all $v_1, \, v_2$. Therefore  and so
$$
d\lambda\left(\frac{D}{dt}d\phi^t\circ (d\phi^t)^{-1} (v_1), v_2\right) + d\lambda\left(v_1, \frac{D}{dt}d\phi^t\circ (d\phi^t)^{-1}(v_2)\right) = 0.
$$
Since we have $DR_\lambda(p)(v) =\frac{D}{dt}d\phi^t(p)\circ (d\phi^t)^{-1}(v)$ by definition,
we obtain
$$
d\lambda(DR_\lambda(p)(v_1), v_2) + d\lambda(v_1,DR_\lambda(p)(v_2)) = 0.
$$
In terms of the metric $g= d\lambda(\cdot, J_\xi\cdot)$, this can be rewritten as
$$
- g(DR_\lambda(p)(v_1),J_\xi v_2) + g(Jv_1, DR_\lambda(p)(v_2)) = 0.
$$
Hence we have
$$
g(DR_\lambda(p)(v_1),J_\xi v_2) = - g(v_1, JDR_\lambda(p)(v_2))
$$
By setting $v_2' = J_\xi v_2$, this is equivalent to
$$
g(DR_\lambda(p)(v_1),v_2') = g(v_1, JDR_\lambda(p)J(v_2')).
$$
This proves $ DR_\lambda(p)^\dagger = J(p)DR_\lambda(p)J(p)$.
\end{proof}

Using this we derive

\begin{lem} For any $\eta \in \ker (d_{(\gamma, T)}\Phi)^\dagger$, we have
$\eta(t) \perp R_\lambda(\gamma(t))$ for all $t \in [0,1]$.
\end{lem}
\begin{proof} By the hypothesis $R_0 \cap R_1 = \emptyset$,  $\gamma$ cannot be
a closed orbit and there exists an interval $I$ open in $[0,1]$ such that
$$
\# \gamma^{-1}(t) \equiv 1 \mod 2
$$
for all $t \in I$. Note that we can choose $I$ so that it is either $I = [0, b)$ or $I = (a, 1]$.
Then \eqref{eq:cokernel-2} implies $\eta(t) \perp R_\lambda(\gamma(t))$ for all $t \in I$.

On the other hand, using \eqref{eq:cokernel-1}, we compute
$$
\frac{d}{dt} \langle R_\lambda(\gamma(t)), \eta(t) \rangle = \langle R_\lambda(\gamma(t)),
\frac{D \eta}{dt} \rangle = \langle R_\lambda(\gamma(t)), T DR_\lambda(\gamma)^\dagger \eta(t) \rangle
$$
for all $t \in [0,1]$, with respect to the contact triad connection $\nabla$.
We write
$$
\eta(t) = \eta^\pi(t) + \lambda(\eta(t)) R_\lambda(\gamma(t))
$$
and recall $\nabla_{R_\lambda} (\xi) \subset (\xi)$ and $\nabla_{R_\lambda} R_\lambda = 0$.
Then, substituting $DR_\lambda(\gamma)^\dagger = J DR_\lambda(\gamma) J$ on $\xi$ and noting
$DR_\lambda(\gamma) (\xi) \subset \xi$, we derive that the operator $DR_\lambda(\gamma)^\dagger$
preserves the splitting
$TM = \xi \oplus \R \langle R_\lambda \rangle$.
Then we rewrite
\beastar
\langle R_\lambda(\gamma(t)), T DR_\lambda(\gamma)^\dagger \eta(t) \rangle
& = & \langle R_\lambda(\gamma(t)), T \lambda(\eta) R_\lambda(\gamma(t)) \rangle\\
& = & T \lambda(\eta)  = T \langle R_\lambda, \eta \rangle.
\eeastar
In conclusion, the function $g(t): = \langle R_\lambda(\gamma(t)), \eta (t) \rangle$ satisfies
the linear 1-st order ODE
\be\label{eq:dgdt}
\dot g(t) - T g (t)= 0.
\ee
On the other hand, on $I$, we have $\lambda(\eta(t)) = \langle R_\lambda(\gamma(t)), \eta(t) \rangle = 0$ for
all $t \in I$. This implies $g(t) \equiv 0$ for all $t \in I$. Since $g$ satisfies \eqref{eq:dgdt}, this implies
$g(t) = 0$ for all $t \in [0,1]$, which finishes the proof of the lemma.
\end{proof}

This lemma implies $\eta(t) \in \xi_{\gamma(t)}$. Then
using the identity
$$
DR_\lambda(\gamma)^\dagger = J(\gamma)DR_\lambda(\gamma)J(\gamma)
$$
on $\xi$, it follows
$$
\frac{D\eta}{d t}  + T\,J DR_\lambda(\gamma) J \eta = 0, \quad \eta(t) \in \xi_{\gamma(t)}, \,  \eta(0)
\in N_{\gamma(0)}R_0 \cap \xi_{\gamma(0)}
$$
i.e.,
\be\label{eq:adjoint}
\frac{DJ\eta}{d t} - T DR_\lambda(\gamma) J \eta = 0, \quad \eta(t) \in \xi_{\gamma(t)}, \, \eta(0)
\in N_{\gamma(0)}R_0 \cap \xi_{\gamma(0)}.
\ee
We consider the family
$$
v(t) = (d\phi^{Tt})^{-1} (\gamma(t)) J(\gamma)\eta(t)) \in \xi_p
$$
and differentiate
\beastar
\frac{d v}{dt}& = &(d\phi^{Tt})^{-1} \frac{D (J \eta)}{dt} - T (d\phi^{Tt})^{-1} \frac{D (d\phi^{Tt})}{dt} (d\phi^t)^{-1}(J \eta(t))\\
& = & (d\phi^{Tt})^{-1} \left(\frac{D (J\eta)}{dt} - T \frac{D (d\phi^{Tt})}{dt} (d\phi^{Tt})^{-1}(J \eta(t))\right).
\eeastar
But by definition, we have
$$
\frac{D (d\phi^{Tt})}{dt} (d\phi^{Tt})^{-1} = DR_\lambda(\gamma(t))
$$
and hence we obtain $\frac{d v}{dt} \equiv 0$. Therefore we have
\be\label{eq:eigenvector}
v(1) = v(0),\,  \text{\rm i.e.,} \,  (d\phi^T)^{-1}(J \eta(1)) = J \eta(0).
\ee
Since $\eta(0) \in N_{\gamma(0)}R_0 \cap \xi_{\gamma(0)}$ and
$\eta(1) \in N_{\gamma(1)}R_0 \cap \xi_{\gamma(1)}$, we have
$$
J\eta(0) \in T_{\gamma(0)} R_0, \quad J\eta(1) \in T_{\gamma(1)}R_1.
$$
Then \eqref{eq:eigenvector} shows that
$$
\ker DR_\lambda(\gamma)^\dagger = 0
$$
is equivalent to
$$
d\phi^T(T_{\gamma(0)}R_0) \cap T_{\gamma(1)} R_1 = \{0\}
$$
which is equivalent to saying that $\gamma$ is a nondegenerate Reeb chord
from $R_0$ to $R_1$. The converse also holds by reading the above proof backwards.
This finishes the proof of Proposition \ref{prop:nondegeneracy-chords}.
\end{proof}

Motivated by Proposition \ref{prop:nondegeneracy-chords},
we now consider the full derivative $d\Phi$.
It remains to compute $d_\lambda \Phi$. For this purpose, we recall the defining equation
of $R_\lambda$:
$$
X \rfloor \lambda = 1, \, X \rfloor d\lambda = 0
$$
Consider the small perturbation $\lambda_\e = \lambda + \e \mu$ and write the
corresponding Reeb vector field by $R_{\lambda_\e} = R_\lambda + \e Y (\mod o(\e))$.
Then we have
$$
(R_\lambda + \e Y) \rfloor (\lambda + \e \mu) = 1, \, (R_\lambda + \e Y)
\rfloor d(\lambda + \e \mu)=0 \quad \mod o(\e).
$$
By collecting the terms of order $\e$, we obtain
$$
Y \rfloor \lambda + R_\lambda \rfloor \mu = 0, \, R_\lambda \rfloor d\mu
+ Y \rfloor d\lambda = 0.
$$
Hence the variation $\delta_\lambda R_\lambda (\mu) =: Y_\mu$ is uniquely determined by the
equation
\be\label{eq:Ymu}
Y_\mu \rfloor \lambda = - R_\lambda \rfloor \mu, \quad  Y_\mu \rfloor d\lambda = -  R_\lambda \rfloor d\mu.
\ee
Now we are ready to study $d_\lambda\Phi(\mu)$. We have
$$
d_\lambda\Phi(\mu) =  -T\, Y_\mu(\gamma).
$$
Therefore if $\eta \in \coker (d\Phi(\gamma))$, we must have
$$
\int_0^1 \langle Y_\mu (\gamma(t)), \eta(t) \rangle\, dt = 0
$$
for all $\mu$. Since $\eta(t) \in \xi_{\gamma(t)}$, we have
$\langle Y_\mu (\gamma(t)), \eta(t) \rangle = d\lambda(Y_\mu(t),J_\xi(\gamma(t))\eta(t))$.
And \eqref{eq:Ymu} implies
$$
d\lambda(Y_\mu(t),J_\xi(\gamma(t))\eta(t)) = - d\mu(R_\lambda(\gamma),J_\xi(\gamma)\eta).
$$
Therefore we have
\be\label{eq:intvanishing}
0= \int_0^1 \langle Y_\mu (\gamma(t)), \eta(t) \rangle\, dt
= -\int_0^1 d\mu(R_\lambda(\gamma(t)),J_\xi(\gamma(t))\eta(t))\, dt
\ee
for any one-form $\mu$. Now the following lemma will
finish the proof.

\begin{lem} Let $q \in M$ and consider $\xi_q \subset T_q M$. Denote by
$\{R_\lambda(q)\}^\perp \subset T_q^*M$ the annihilator of $R_\lambda(q)$.
Then we have
$$
\{d\mu(R_\lambda(q),\cdot) \in T^*_q M \mid d\mu \in \Gamma(S^2(T_q^*M))
\} = \{R_\lambda(q)\}^\perp \cong \xi_q^*.
$$
\end{lem}
\begin{proof} Obviously we have
$$
\{d\mu(R_\lambda(q),\cdot) \in T^*_q M \mid d\mu \in \Gamma(S^2(T_q^*M))
\} \subset \{R_\lambda(q)\}^\perp.
$$
For the opposite inclusion, it is enough to note in terms of
local coordinates that for any nonzero vector
$v \in \R^n$, the map
$$
A \mapsto Av; \, \quad \Lambda^2(\R^n) \to \R^n
$$
is surjective. Here $\Lambda^2(\R^n)$ is the set of skew-symmetric
matrices and $n \geq 2$. This finishes the proof.
\end{proof}

Once we have this lemma, we can conclude \eqref{eq:intvanishing}
and the unique continuation for the equation \eqref{eq:adjoint} imply $\eta \equiv 0$.
This finishes the proof of the theorem. \qed

For the later purpose, we also need the following theorem.

\begin{thm}\label{thm:opneN} Let
 $\Cont^{\text{\rm reg}}(M,\xi;R_0,R_1;< N)$ be the set
of $\lambda$'s such that all $\lambda$-Reeb chords $\gamma$ from $R_0$ to $R_1$
with $\CA_\lambda(\gamma) < N$. Then
it is open in $\Cont(M,\xi)$ for each given $N > 0$.
\end{thm}
\begin{proof} Consider the two projection
$$
\xymatrix{& \ar[dl]_{\Pi_1} \Phi^{-1}(0) \ar[dr]^{\Pi_3}& \\
\R_+ & &\Cont(M,\xi;R_0,R_1)}
$$
where $\Pi_i = \pi_i|_{\Phi^{-1}(0)}$. We denote
$$
\frak{Reeb}(M,\xi; R_0,R_1;<N): = \Pi_1^{(0,N)}.
$$
Then $\Cont^{\text{\rm reg}}(M,\xi;R_0,R_1;N)$ is the set of regular values of the map
$$
\Pi_3|_{\frak{Reeb}(M,\xi; R_0,R_1;<N)} : \frak{Reeb}(M,\xi; R_0,R_1;< N) \to \Cont(M,\xi;R_0,R_1).
$$
Now let $\lambda \in \Cont^{\text{\rm reg}}(M,\xi;R_0,R_1;N)$. Then the set
$$
\frak{Reeb}(M,\xi; R_0,R_1;<N) \cap \Pi_3^{-1}(\lambda)
$$
is compact.
Therefore by the tube lemma, there exists an open neighborhood $\CV$ of $\lambda$ in $\Cont(M,\xi;R_0,R_1)$
such that all $\lambda'$-Reeb chords in $(\Pi_3|_{\frak{Reeb}(M,\xi; R_0,R_1;<N})^{-1}(\CV)$ are nondegenerate
and hence $\CV \subset \Cont^{\text{\rm reg}}(M,\xi;R_0,R_1;<N)$. This finishes the proof.
\end{proof}

\begin{thm} \label{thm:Reeb-chords-MB}
Let $(M,\xi)$ be a contact manifold. Let  $(R_0,R_1)$ be a pair of Legendrian submanifolds
allowing the case $R_0 = R_1$.
 For a given contact form $\lambda$, there exists a residual subset of pairs $(R_0,R_1)$
of Legendrian submanifolds such that  all
Reeb chords from $R_0$ to $R_1$ are nondegenerate for $T > 0$ and
Bott-Morse nondegenerate when $T = 0$.
\end{thm}

\subsection{Under the perturbation of boundaries}

In this section, we prove the following generic perturbation problem of the boundary
Legendrian submanifolds by transforming the problem to that of
perturbation of contact forms. Since we gave complete details of the proof of
Theorem \ref{thm:Reeb-chord-lambda}, we will just indicate the differences in the proof of
the following theorem therefrom.

\begin{thm}\label{thm:Reeb-chord-bdy} Let $(M,\xi)$ be a contact manifold,
$\lambda$ a contact form and $R_1 \in {\mathcal Leg}(M,\xi)$. Then there exists a residual subset
$$
R_0 \in{\mathcal Leg}^{\text{\rm reg}}(M,\xi) \subset {\mathcal Leg}(M,\xi)
$$
of Legendrian submanifolds such that for all $R_0 \in {\mathcal Leg}(M,\xi)$ all
Reeb chords from $R_0$ to $R_1$ are nondegenerate for $T > 0$ and Bott-Morse nondegenerate when $T = 0$.
\end{thm}
\begin{proof} This time we consider the fiber bundle $\CL^{1,2}(M;R_1)$ over
$$
(0,\infty) \times {\mathcal Leg}(M,\xi)
$$
whose fiber at $(T,R_0)$ is given by
$$
\CL^{1,2}_{(T,R_0)}(M;R_1) = \CL^{1,2}(M;R_0,R_1).
$$
Then we consider the assignment
$$
\Phi: (T,\gamma, R_0) \mapsto \dot \gamma - T R_\lambda(\gamma)
$$
as a section of the Banach vector bundle
$$
\CL^2(M;R_1) \to \CL^{1,2}(M;R_1)
$$
whose fiber at $(T,\gamma,R_0)$ is given by the vector space
$$
L^2(\gamma^*TM, T_{\gamma(0)}R_0, T_{\gamma(1)}R_1).
$$
Now we consider a perturbation of $R_0$ under the contact isotopy of the type
$$
\psi(R_0), \quad \psi = \psi_H^1 \in \Cont_0(M,\xi).
$$
Then for given $\gamma \in \CL^{1,2}(M;\psi(R_0),R_1)$, the composition path
$$
\widetilde \gamma(t): = (\psi_H^t)^{-1}(\gamma(t))
$$
satisfies the perturbed equation
$$
\begin{cases}
\dot {\widetilde \gamma}(t) = - X_H(\widetilde \gamma(t)) + (\psi_H^t)^*R(\widetilde \gamma(t)\\
\widetilde \gamma(0) \in R_0, \quad \widetilde \gamma(1) \in R_1
\end{cases}
$$
with \emph{fixed} boundary condition. Then we can duplicate the proof of Theorem \ref{thm:Reeb-chord-lambda}
by replacing perturbation of $\lambda$ by that of Hamiltonian $H$ above
with almost same kind computation and so we omit the details.
This then finishes the proof of
Theorem \ref{thm:Reeb-chord-bdy}.
\end{proof}

\bibliographystyle{amsalpha}

\bibliography{biblio2}

\end{document}